\documentclass[journal,10pt]{IEEEtran}
\ifCLASSINFOpdf
\else
\fi
\usepackage{comment} 

\usepackage{amsmath,graphicx}

\usepackage{pstricks,graphicx}
\usepackage{setspace}
\usepackage{psfrag}
\usepackage{amssymb}
\usepackage{amsmath}
\usepackage{amsthm}

\newtheorem{remark}{Remark}
\newtheorem{assumption}{Assumption}
\newtheorem{lemma}{Lemma}

\newtheorem{theorem}{Theorem}
\newtheorem{proposition}{Proposition}

\begin{document}

%\onecolumn

%\doublespacing 

% paper title
% can use linebreaks \\ within to get better formatting as desired
\title{On Relationship between Primal-Dual Method of Multipliers and Kalman Filter}

\graphicspath{{figures/}}
% author names and IEEE memberships
% note positions of commas and nonbreaking spaces ( ~ ) LaTeX will not break
% a structure at a ~ so this keeps an author's name from being broken across
% two lines.
% use \thanks{} to gain access to the first footnote area
% a separate \thanks must be used for each paragraph as LaTeX2e's \thanks
% was not built to handle multiple paragraphs
%

\author{Guoqiang~Zhang, W. Bastiaan Kleijn and Richard Heusdens
\thanks{G.~Zhang is with the Center of Audio, Accoustic and Vibration (CAAV), School of Electrical and Data Engineering, University of Technology, Sydney, Australia. Email: {guoqiang.zhang@uts.edu.au}}
\thanks{W.~B.~Kleijn is with is the school of Engineering and Computer Science, Victoria University of Wellington, New Zealand.
Email: {bastiaan.kleijn@ecs.vuw.ac.nz}}
\thanks{R.~Heusdens is with the Department of Microelectronics, Circuits and Systems group, Delft University of Technology, Delft, the Netherlands. Email: {r.heusdens@tudelft.nl}}
}

\maketitle
\begin{abstract}
Recently the primal-dual method of multipliers (PDMM), a novel distributed optimization method, was proposed for solving a general class of decomposable convex optimizations over graphic models. In this work, we first study the convergence properties of PDMM for decomposable quadratic optimizations over tree-structured graphs. We show that with proper parameter selection, PDMM converges to its optimal solution in finite number of iterations. We then apply PDMM for the causal estimation problem over a statistical linear state-space model.  We show that PDMM and the Kalman filter have the same update expressions, where PDMM can be interpreted as solving a sequence of quadratic optimizations over a growing chain graph.  
\end{abstract}

% IEEEtran.cls defaults to using nonbold math in the Abstract.
% This preserves the distinction between vectors and scalars. However,
% if the journal you are submitting to favors bold math in the abstract,
% then you can use LaTeX's standard command \boldmath at the very start
% of the abstract to achieve this. Many IEEE journals frown on math
% in the abstract anyway.

% Note that keywords are not normally used for peerreview papers.

% For peer review papers, you can put extra information on the cover
% page as needed:
% \ifCLASSOPTIONpeerreview
% \begin{center} \bfseries EDICS Category: 3-BBND \end{center}
% \fi
%
% For peerreview papers, this IEEEtran command inserts a page break and
% creates the second title. It will be ignored for other modes.
\IEEEpeerreviewmaketitle

\section{Introduction}
In recent years, distributed optimization has drawn increasing attention due to the demand for big-data processing and easy access to ubiquitous computing units (e.g., a computer, a mobile phone or a sensor equipped with a CPU). The basic idea is to have a set of computing units collaborate with each other in a distributed way to complete a complex task. Popular applications include telecommunication \cite{Richardson08Coding,xiaoqiang13ADMMLDPC}, wireless sensor networks \cite{Boyd06gossip}, cloud computing and machine learning \cite{Sontag11ML}. The research challenge is on the design of efficient and robust distributed optimization algorithms for those applications.

Among various distributed optimization methods, the alternating direction method of multipliers (ADMM) is probably the most popular method being studied and applied in practice. It is aimed at solving the following decomposable convex optimization problem:  
\begin{align}
\min_{\boldsymbol{x},\boldsymbol{z}} f(\boldsymbol{x})+g(\boldsymbol{z}) \textrm{ s.t. } \boldsymbol{A}\boldsymbol{x}+\boldsymbol{B}\boldsymbol{z}= \boldsymbol{c}, 
\label{equ:problemTwoNode}
\end{align}
where $\textrm{s.t.}$ stands for ``subject to", and the two matrices $(\boldsymbol{A},\boldsymbol{B})$ and the vector $\boldsymbol{c}$ are known a priori. Many practical optimization problems can be reformulated or reduced to the form (\ref{equ:problemTwoNode}), such as  the resource allocation in wireless networks \cite{Joshi13ADMM},  compressive sensing \cite{Yang11ADMM,Zhu17ADMM},  the image denoising problem \cite{Figueiredo10ADMM}, and telecommunication problems \cite{Joshi13ADMMMISO, xiaoqiang13ADMMLDPC}.

Recent research activities on ADMM have focused on improving its convergence speed. The work of \cite{Teixeira16ADMM} considers the optimal parameter selection of ADMM for the quadratic consensus problem. In \cite{Ghadimi15ADMM}, the authors study the optimal parameter selection for  quadratic programming.  The work in \cite{Giselsson17ADMM} considers the optimal parameter selection when either $f(\boldsymbol{x})$ or $g(\boldsymbol{z})$ in (\ref{equ:problemTwoNode}) is a quadratic function.   

Recently, we have proposed the primal-dual method of multipliers (PDMM) \cite{Zhang16PDMM} for solving a general class of decomposable optimization problems over an undirected graphical model $G=\{\mathcal{V},\mathcal{E}\}$ that takes the form:
\begin{align}
\min \sum_{i\in \mathcal{V}} f_i(\boldsymbol{x}_i) \textrm{ s.t. } \boldsymbol{A}_{i|j}\boldsymbol{x}_i + \boldsymbol{A}_{j|i}\boldsymbol{x}_j = \boldsymbol{c}_{ij} \; \forall (i,j)\in \mathcal{E}, \label{equ:problemMultiNode}
 \end{align}
 where the subscript $i|j$ indicates that the matrix $\boldsymbol{A}_{i|j}$ operates on $\boldsymbol{x}_i$ for the equality constraint over the edge $(i,j)\in \mathcal{E}$. In principle, ADMM can also be applied to solve (\ref{equ:problemMultiNode}) by introducing a set of auxiliary variables (see \cite{Zhang16PDMM} or \cite{Shi14ADMM} for a reformulation of (\ref{equ:problemMultiNode}) into (\ref{equ:problemTwoNode})). PDMM is carefully designed by avoiding the auxiliary variables as required in ADMM. %Empirical studies on  a few problems (e.g., \cite{Zhang16PDMM,Heming15Thesis}) show that PDMM is more efficient than ADMM.% which may be partially due to avoidance of the auxiliary variables. 

In the literature, PDMM has been applied successfully for solving a number of practical problems. The work of \cite{Heming15Thesis} investigates the efficiency of ADMM and PDMM for distributed dictionary learning. In \cite{xiaoqiang16BiADMM}, we have used both ADMM and PDMM for training a support vector machine (SVM). In the above examples it is found that PDMM outperforms ADMM in terms of convergence rate. Furthermore, PDMM has also been successfully applied to distributed speech enhancement over a wireless microphone network (see  \cite{Tavakoli16PDMM, Tavakoli17PDMM, Connor16PDMM, Sherson16LCMV_conf}).  The basic principle in those works is to reformulate different speech enhancement problems as decomposable convex optimizations of form (\ref{equ:problemMultiNode}), allowing the application of PDMM.   

Theoretical convergence analyses of PDMM are provided in \cite{Zhang16PDMM} and \cite{Sherson17PDMM}. The work of \cite{Zhang16PDMM} uses the variational inequality to conduct the analysis while \cite{Sherson17PDMM} makes use of monotone operator theory \cite{Ryu16Mono}, \cite{Eckstein_phdThesis} to interpret PDMM as a Peaceman-Rachford splitting method.  %Convergence analysis is then provided in \cite{Sherson17PDMM}  for different functional properties of the convex optimization (\ref{equ:problemMultiNode}).

Differently from PDMM, which is motivated from distributed optimization, the Kalman filter is designed from the perspective of estimation theory. The filter was introduced in the 1960s by Kalman and Bucy \cite{Kalman60}, \cite{Kalman61newresults}.  Given a stochastic linear state-space model driven by Gaussian noise, the filter is able to perform causal optimal estimation of the current hidden state from the past observations or measurements efficiently \cite{KailathBook}. The filter has found many successful applications since its introduction, such as object tracking, state and parameter estimation for control or diagnosis, and signal processing (see \cite{Auger13Kalman} for an overview). 

In this work, we attempt to characterize the relationship between PDMM and the Kalman filter. To do so, we first study the performance of PDMM for a distributed quadratic optimization over a tree structured graph. We show that by proper parameter selection, the method converges to the optimal solution in finite number of iterations for a fixed graph.  

%After that,  we consider the application of PDMM to the causal estimation problem over a state-space model as the Kalman filter does. Interestingly,  we find that PDMM and the Kalman filter have the same update expressions, which confirms the generality of PDMM.    

We then consider applying PDMM to the causal estimation over a state-space model as the Kalman filter does. It can be shown that when the state-space model is driven by Gaussian noise, the estimation problem is equivalent to a maximum likelihood  (ML) problem. By associating each time step of the state-space model with a node, the ML problem becomes a decomposable quadratic optimization over a growing chain graph as the state-space model evolves over time. We show that applying PDMM to the reformulated ML problem with proper parameter selection produces the same updating expressions as the Kalman filter under mild conditions.         

We note that our work is related to but different from the distributed Kalman filter (DKF) \cite{Olfati07DKF}.  The DFK considers the problem that each node in a sensor network receives node-dependent measurements of a common linear state-space model over time. The objective is to estimate the common hidden state cooperatively by the sensors in a distributed manner over time. On the other hand, PDMM is primarily designed for solving the time-invariant problem (\ref{equ:problemMultiNode}).  

%Are there obvious generalizations of Kalman that the quadratic-problem tree version of PDMM can handle but Kalman cannot? At first sight you would  think that tracking a projectile that splits into many small projectiles is one, but at second sight, not so. The multi-track just means that the Kalman formulation  splits into parallel independent components.

%Interestingly,  we find that PDMM and the Kalman filter have the same update expressions, which confirms the generality of PDMM.    

The remainder of the paper is organized as follows. Section~\ref{sec:pre} introduces notations and formally defines the convex optimization (\ref{equ:problemMultiNode}) over a graphic model as well as the quadratic optimization problem. Section~\ref{sec:PDMM} describes both synchronous PDMM and asynchronous PDMM  for solving the general convex problem (\ref{equ:problemMultiNode}). In Section~\ref{sec:finiteConvergence}, we show that with proper parameter selection, PDMM converges in finite steps for the considered quadratic optimization problem. Based on the results in Section~\ref{sec:finiteConvergence},  we show in Section~\ref{sec:RelationKalmanFilter} that the Kalman filter and asynchronous PDMM share the same updating expressions. Finally, we draw conclusions and discuss future works in Section~\ref{sec:conclusion}. 

%Our motivation is to gain insights from this problem on how to select optimal parameters of PDMM for general objective functions and/or loopy graphs. m all the leaf nodes towards the root node, PDMM converges in finite steps. Further the update at the root node reaches optimality first while those of the leaf nodes reach optimality last. We also study the relationships between PDMM and Kalman filter for the estimation problem over a linear statistical model. 
 
%Through experiment, we find that if the scalar parameters $\{{\gamma}_{ij}| (i,j)\in E \}$ are selected properly, PDMM converges in finite steps (at most $2|V|$ steps) regardless of any initialization of $({x},{\beta})$. To do so, one node is chosen from the graph and taken as the root note $r\in V$. The scalar parameters can then be  determined recursively starting from every leaf node towards the root node $r$. Suppose the parameter $\gamma_{ij}$ is to be determined, where node $j$ is closer to the root node $r$ than node $i$.  Then $\gamma_{ij}$ is computed as $\gamma_{ij}=1+\sum_{m\in N_i, m\neq j} \gamma_{mi}$. Empirical results show that the updates of PDMM at root node $r$  reaches optimality first, while those of the leaf nodes reach optimality last. The above preliminary results suggest that distributed algorithms motivated from the framework of convex optimization can have the same convergence speed as the inference methods motivated from machine learning.  

\vspace{-1mm}
\section{Problem Definition}
\vspace{-1mm}
\label{sec:pre}
\subsection{Notations}
\vspace{-1mm}
We first introduce notations for an undirected graphical model. We denote an undirected graph as $G=(\mathcal{V},\mathcal{E})$, where $\mathcal{V}=\{1,\ldots, m\}$ represents the set of nodes and $\mathcal{E}=\{(i,j)| i, j\in \mathcal{V}\}$ represents the set of undirected edges in the graph, respectively. If $(i,j)\in \mathcal{E}$, node $i$ and $j$ can communicate with each other directly along their  edge $(i,j)$.  We use $\mathcal{N}_i$ to denote the set of all neighboring nodes of node $i$, i.e., $\mathcal{N}_i=\{j|(i,j)\in \mathcal{E}\}$. The notation $\mathcal{N}_{i}/ j$ represents the set of  neighbours of $i$ except neighbour $j$. The undirected graph $G$ will be used for problem definition and presentation of the updating procedure of PDMM.  %Given a graph $G=(\mathcal{V},\mathcal{E})$, only neighboring nodes are allowed to communicate with each other directly. 

We denote a directed acyclic graph (DAG) as $\vec{G}=(\mathcal{V},\vec{\mathcal{E}})$, where $\vec{\mathcal{E}}=\{[i,j]| i, j\in \mathcal{V}\}$ represents the set of all directed edges. The directed edge $[i,j]$ indicates that node $i$ can reach node $j$ through the edge but not in reverse order.  The graph being acyclic implies that there exists no path starting at any node $i$ and following a sequence of directed edges that eventually loops back to $i$ again. The DAG $\vec{G}$ will be used for parameter selection of PDMM for the quadratic optimization later on. 

Next we introduce notations for mathematical description in the remainder of the paper. We use bold lower-case letters to denote vectors and bold capital letters to denote matrices. We use $\boldsymbol{I}$ to denote an identity matrix. The notation $\boldsymbol{M}\succeq 0$ (or $\boldsymbol{M}\succ 0$) represents a symmetric positive semi-definite matrix (or a symmetric positive definite matrix). The superscript $(\cdot)^T$ represents the transpose operator while $(\cdot)^{-1}$ represents matrix inversion. Given a vector $\boldsymbol{y}$, we use $\|\boldsymbol{y}\|$ to denote its $l_2$ norm. Furthermore, $\|\boldsymbol{y}\|_{\boldsymbol{M}}^2$ represents the computation $\boldsymbol{y}^T\boldsymbol{M}\boldsymbol{y}$. %Finally, we use $E[\cdot]$ and $\textrm{Cov}[\cdot]$ to denote the expectation and computation of a covariance matrix of a random vector, respectively.  

\vspace{-2mm}
\subsection{Decomposable convex optimization}
\vspace{-1mm}

With the notation $G=(\mathcal{V},\mathcal{E})$ for an undirected graph, problem (\ref{equ:problemMultiNode}) can be formally defined as
\begin{align}
\min_{\boldsymbol{x}} \sum_{i\in \mathcal{V}} f_i(\boldsymbol{x}_i) \textrm{ s. t. } \boldsymbol{A}_{i|j}\boldsymbol{x}_i + \boldsymbol{A}_{j|i}\boldsymbol{x}_{j}=\boldsymbol{c}_{ij} \quad (i,j)\in \mathcal{E}, 
\label{equ:problemMultiNode_re}
\end{align}
where each function $f_i: \mathbb{R}^{n_i}\rightarrow \mathbb{R}\cup \{+\infty\} $ is assumed to be closed, proper and convex, and $\boldsymbol{x}=[\boldsymbol{x}_1^T,\boldsymbol{x}_2^T,\ldots,\boldsymbol{x}_m^T]^T$ is referred to as the \emph{primal} variable. For every edge $(i,j)\in \mathcal{E}$, we have $(\boldsymbol{c}_{ij},\boldsymbol{A}_{i | j}, \boldsymbol{A}_{j| i})\in (\mathbb{R}^{n_{ij}},\mathbb{R}^{n_{ij}\times n_i},\mathbb{R}^{n_{ij}\times n_j})$. The vector $\boldsymbol{x}$ is thus of dimension $n_{\boldsymbol{x}}=\sum_{i\in \mathcal{V}} n_i$. In general, $\boldsymbol{A}_{i | j}$ and $\boldsymbol{A}_{j | i}$ are two different matrices. The matrix $\boldsymbol{A}_{i | j}$ operates on $\boldsymbol{x}_i$ in the linear constraint of edge $(i,j)\in\mathcal{E}$. 

The Lagrangian function for (\ref{equ:problemMultiNode_re}) can be constructed as
%\begin{align}
%L_p(\boldsymbol{x},\boldsymbol{\delta})=\sum_{(i,j)\in E}\boldsymbol{\delta}_{ij}^{\top}(\boldsymbol{c}_{ij}-\boldsymbol{A}_{i %j}\boldsymbol{x}_i-\boldsymbol{A}_{j i}\boldsymbol{x}_j)+\hspace{-1mm}\sum_{i\in V} f_i(\boldsymbol{x}_i), \label{equ:primalLag}
%\end{align}
\begin{align}
\hspace{-1.2mm}L(\boldsymbol{x},\boldsymbol{\delta})\hspace{-0.8mm}= \hspace{-0.5mm}\sum_{i\in \mathcal{V}}\hspace{-0.4mm} f_i(\boldsymbol{x}_i)\hspace{-0.5mm}+\hspace{-2mm}\sum_{(i,j)\in \mathcal{E}}\hspace{-1.5mm}\boldsymbol{\delta}_{ij}^{T}(\boldsymbol{c}_{ij}\hspace{-0.5mm}-\hspace{-1mm}\boldsymbol{A}_{i | j}\boldsymbol{x}_i\hspace{-0.5mm}-\hspace{-1mm}\boldsymbol{A}_{j  | i}\boldsymbol{x}_j\hspace{-0.3mm}), \label{equ:primalLag}
\end{align}
where $\boldsymbol{\delta}_{ij}$ is the Lagrangian multiplier (or the dual variable) for the corresponding edge constraint in (\ref{equ:problemMultiNode_re}), and the vector $\boldsymbol{\delta}$ is obtained by stacking all the dual variables $\boldsymbol{\delta}_{ij}$, $(i,j)\in \mathcal{E}$. Therefore, $\boldsymbol{\delta}$ is of dimension $n_{\boldsymbol{\delta}}=\sum_{(i,j)\in \mathcal{E}}n_{ij}$. The Lagrangian function is convex in $\boldsymbol{x}$ for fixed $\boldsymbol{\delta}$, and concave in $\boldsymbol{\delta}$ for fixed $\boldsymbol{x}$. We make the following assumption for the convex optimization (\ref{equ:problemMultiNode_re}):

\begin{assumption}
There exists a saddle point $(\boldsymbol{x}^{\star},\boldsymbol{\delta}^{\star})$ to the Lagrangian function $L(\boldsymbol{x},\boldsymbol{\delta})$ such that for all $\boldsymbol{x}\in \mathbb{R}^{n_{\boldsymbol{x}}}$ and $\boldsymbol{\delta}\in \mathbb{R}^{n_{\boldsymbol{\delta}}}$ we have
\begin{align}
L(\boldsymbol{x}^{\star},\boldsymbol{\delta}) \leq L(\boldsymbol{x}^{\star},\boldsymbol{\delta}^{\star})\leq L(\boldsymbol{x},\boldsymbol{\delta}^{\star}) \nonumber.
\end{align}
Or equivalently, the following optimality (KKT) conditions hold for $(\boldsymbol{x}^{\star},\boldsymbol{\delta}^{\star})$:
\begin{align}
\sum_{j\in \mathcal{N}_i}\boldsymbol{A}_{i | j}^T\boldsymbol{\delta}_{ij}^{\star} \in \partial f_i(\boldsymbol{x}_i^{\star}) &\quad \forall i\in \mathcal{V} \label{equ:KKT_prim_1} \\
\boldsymbol{A}_{j | i}\boldsymbol{x}_j^{\star}+\boldsymbol{A}_{i | j}\boldsymbol{x}_i^{\star}=\boldsymbol{c}_{ij} &\quad \forall (i,j)\in \mathcal{E}. \label{equ:KKT_prim_2}
\end{align}\vspace{-4mm}
\label{assumption:KKT}
\end{assumption}
 
 A saddle point $(\boldsymbol{x}^{\star},\boldsymbol{\delta}^{\star})$ to the Lagrangian function $L(\boldsymbol{x},\boldsymbol{\delta})$  provides an optimal solution $\boldsymbol{x}^{\star}$ to the original problem (\ref{equ:problemMultiNode_re}). There might be many saddle points to the Lagrangian function. 
 
 \vspace{-2mm}
\subsection{Decomposable quadratic optimization}
\vspace{-1mm}

As a special case of (\ref{equ:problemMultiNode_re}), a decomposable quadratic optimization can be represented as
\begin{align}
\boldsymbol{x}^{\star}=&\min_{\boldsymbol{x}} \sum_{i\in \mathcal{V}} \left(f_i(\boldsymbol{x}_i)=\frac{1}{2}\boldsymbol{x}_i^T\boldsymbol{\Sigma}_i \boldsymbol{x}_i-\boldsymbol{a}_i^T\boldsymbol{x}_i\right) \nonumber 
\end{align}
\begin{align}
& \textrm{ s. t. } \boldsymbol{A}_{i|j}\boldsymbol{x}_i + \boldsymbol{A}_{j|i}\boldsymbol{x}_{j}=\boldsymbol{c}_{ij} \quad (i,j)\in \mathcal{E}, 
\label{equ:quadOptimization}
\end{align}
where $\boldsymbol{\Sigma}_i$, $i\in \mathcal{V}$, is a symmetric positive semi-definite matrix (i.e., $\boldsymbol{\Sigma}_i \succeq 0$).  In this case, the optimality conditions (\ref{equ:KKT_prim_1})-(\ref{equ:KKT_prim_2}) would become a set of equations. We investigate the quadratic optimization (\ref{equ:quadOptimization}) in Section~\ref{sec:finiteConvergence}.

\section{Primal-Dual Method of Multipliers}
\label{sec:PDMM}
In this section, we present the updating procedure of PDMM for the decomposable convex optimization (\ref{equ:problemMultiNode_re}). We first introduce synchronous PDMM, where all the variables are updated simultaneously per iteration. We then introduce asynchronous PDMM, where only the variables of a selected node are updated per iteration. Finally, we simplify the updating procedure of synchronous and asynchronous PDMM into the so-called \emph{message-passing} framework. 

\subsection{Synchronous PDMM}
Before presenting the updating procedure for (\ref{equ:problemMultiNode_re}), we first introduce a set of dual variables required for the method to work. Let $\boldsymbol{\lambda}_{i|j}$ and $\boldsymbol{\lambda}_{j|i}$ be two dual variables for every edge constraint $\boldsymbol{A}_{i | j}\boldsymbol{x}_i + \boldsymbol{A}_{j | i}\boldsymbol{x}_{j}=\boldsymbol{c}_{ij}$, which are actually two copies of the dual variable $\boldsymbol{\delta}_{ij}$ in (\ref{equ:primalLag}). The variable $\boldsymbol{\lambda}_{i|j}$ is owned by and updated at node $i$ and is related to neighboring node $j$. We denote by $\boldsymbol{\lambda}_i$ the concatenation of all $\boldsymbol{\lambda}_{i|j}$, $j\in \mathcal{N}_i$.  Therefore each node $i$ carries both a primal variable $\boldsymbol{x}_i$ and a dual variable $\boldsymbol{\lambda}_i$.  Similarly to $\boldsymbol{x}$, we let $\boldsymbol{\lambda}=[\boldsymbol{\lambda}_1^T,\ldots,\boldsymbol{\lambda}_{m}^T ]^T$. 

Synchronous PDMM updates $\boldsymbol{x}$ and $\boldsymbol{\lambda}$ simultaneously per iteration by performing node-oriented computation. At iteration $k$, each $i$ computes a new estimate $\hat{\boldsymbol{x}}_{i}^{k+1}$ by locally solving a small-size optimization problem. In doing so, the neighboring estimates $\{\hat{\boldsymbol{x}}_j^{k}|j\in \mathcal{N}_i\}$ and $\{\hat{\boldsymbol{\lambda}}_{j|i}^{k}|j\in \mathcal{N}_i\}$ from last iteration are utilized. Once $\hat{\boldsymbol{x}}_{i}^{k+1}$ is obtained, the new estimate of each dual variable $\hat{\boldsymbol{\lambda}}_{i|j}^{k+1}$ is computed. The update expressions for  $\hat{\boldsymbol{x}}_{i}^{k+1}$ and  $\hat{\boldsymbol{\lambda}}_{i}^{k+1}$ can be represented as \cite{Zhang16PDMM}
\begin{align}
\hspace{-2mm}\hat{\boldsymbol{x}}_i^{k+1}\hspace{-1.5mm}=&\arg\min_{\boldsymbol{x}_i}\hspace{-1mm}\Bigg[
\hspace{-0.8mm}\sum_{j\in \mathcal{N}_i}\hspace{-0.8mm}\frac{1}{2}\left\|\boldsymbol{A}_{i | j}\boldsymbol{x}_i+\boldsymbol{A}_{j |  i}\hat{\boldsymbol{x}}_j^{k}-\boldsymbol{c}_{ij}\right\|_{\boldsymbol{P}_{ij}^{-1}}^2 \nonumber\\
&\hspace{10mm}-\hspace{-0.4mm}\boldsymbol{x}_i^{T}\Bigg(\hspace{-0.5mm}\sum_{j\in \mathcal{N}_i}\hspace{-1.5mm}\boldsymbol{A}_{i | j}^T\hat{\boldsymbol{\lambda}}_{j|i}^{k}\hspace{-0.5mm}\Bigg)\hspace{-0.5mm}+\hspace{-0.5mm}f_i(\boldsymbol{x}_i)\hspace{-0.5mm}\Bigg]  \;\;\;\;\; i\in \mathcal{V} \label{equ:x_updateSyn0} \\
\hspace{-2mm}\hat{\boldsymbol{\lambda}}_{i|j}^{k+1}\hspace{-1.5mm}=& \hat{\boldsymbol{\lambda}}_{j|i}^{k} \hspace{-0.7mm}+\hspace{-0.8mm}
\boldsymbol{P}_{ij}^{-1}(\boldsymbol{c}_{ij}\hspace{-0.7mm}-\hspace{-0.8mm}\boldsymbol{A}_{j | i}\hat{\boldsymbol{x}}_j^k\hspace{-0.7mm}-\hspace{-0.7mm}\boldsymbol{A}_{i | j}\hat{\boldsymbol{x}}_i^{k+1}) \hspace{0.7mm} i\in \mathcal{V}, j\in \mathcal{N}_i, \hspace{-2mm}\label{equ:lambda_updateSyn0}
\end{align}
where $\boldsymbol{P}_{ij}\succ 0$ is a weighting matrix associated with the edge constraint $\boldsymbol{A}_{i | j}\boldsymbol{x}_i + \boldsymbol{A}_{j | i}\boldsymbol{x}_{j}=\boldsymbol{c}_{ij}$.  Differently from $\boldsymbol{A}_{i | j}$ or $\boldsymbol{A}_{j | i}$ which implicitly has a direction, the matrix $\boldsymbol{P}_{ij}$ has no direction, i.e., $\boldsymbol{P}_{ij}=\boldsymbol{P}_{ji}$. We let $\mathcal{P} = \{\boldsymbol{P}_{ij}\succ 0 | (i,j)\in \mathcal{E}\}$ represent the set of all the weighting matrices. 

The update expressions (\ref{equ:x_updateSyn0})-(\ref{equ:lambda_updateSyn0}) hold for the general problem formulation (\ref{equ:problemMultiNode_re}). Therefore, they also hold for the quadratic optimization (\ref{equ:quadOptimization}), of which an analytic expression for $\hat{\boldsymbol{x}}_i^{k+1}$ can  be easily derived by replacing $f_i(\boldsymbol{x}_i)$ with the corresponding quadratic function. 

Convergence analyses for synchronous PDMM have been provided in \cite{Zhang16PDMM} and \cite{Sherson17PDMM}. We present one convergence property from \cite{Zhang16PDMM} in a theorem below. The theorem will be used for analysis later on.  
\begin{theorem}
If the estimate $\hat{\boldsymbol{x}}_i^{k}$ of node $i$ converges to a fixed point $\boldsymbol{x}_i'$ ($\lim_{k\rightarrow \infty}\hat{\boldsymbol{x}}_i^{k}=\boldsymbol{x}_i'$), we have $\boldsymbol{x}_i' =\boldsymbol{x}_i^{\star}$, where $\boldsymbol{x}_i^{\star}$ is the $i$th component of an optimal solution $\boldsymbol{x}^{\star}$ of (\ref{equ:problemMultiNode_re}).
\label{theorem:syn}
\end{theorem}

%It has been shown in \cite{Zhang16PDMM} that as long as the individual objective functions $\{f_i|i\in \mathcal{V}\}$ are differentiable, the above updating procedure (\ref{equ:x_updateSyn0})-(\ref{equ:lambda_updateSyn0}) converges to an optimal solution.  Further if the individual objective functions are strongly convex and Lipschiz contineous,  the methods converge geometrically.  

The set $\mathcal{P}$ of positive definite matrices has a great impact on the convergence speed of PDMM. See \cite{Zhang16PDMM} for empirical evidence on how different parameters affect its convergence speed for the distributed averaging problem. In next section, we  will show that by proper selection of $\mathcal{P}$, the method has a finite-time convergence for the quadratic optimization (\ref{equ:quadOptimization}) over a tree-structured graph.     

We note that for the updating procedure (\ref{equ:x_updateSyn0})-(\ref{equ:lambda_updateSyn0}) to work, the local estimates $\{(\hat{\boldsymbol{x}}_i,\hat{\boldsymbol{\lambda}}_i)|i\in \mathcal{V}\}$ need to be exchanged between neighboring nodes after every iteration. By doing so, every node $i$ gradually captures global information of all the objective functions $\{f_i\}$ and the constraints through its local estimate $(\hat{\boldsymbol{x}}_i,\hat{\boldsymbol{\lambda}}_i)$. 

\begin{remark}
We note that the work in \cite{Sherson17PDMM} focuses on a special structure of the matrix set $\mathcal{P}$, namely $\boldsymbol{P}_{i,j}=\rho \boldsymbol{I}$ for every $\boldsymbol{P}_{i,j} \in \mathcal{P}$, where $\rho>0$. One research direction is to extend the work of \cite{Sherson17PDMM} to the general matrix form $\boldsymbol{P}_{i,j} \succ \boldsymbol{0}$, $(i,j)\in \mathcal{E}$.  
\end{remark}

\subsection{Asynchronous PDMM}
In \textit{asynchrous} PDMM for each iteration, only the variables associated with one node in the graph update their estimates while all other variables keep their estimates fixed. Suppose node $i$ is selected at iteration $k$. We then compute $(\hat{\boldsymbol{x}}_i^{k+1},\hat{\boldsymbol{\lambda}}_i^{k+1})$  by following  (\ref{equ:x_updateSyn0})-(\ref{equ:lambda_updateSyn0}). At the same time, the estimates $(\hat{\boldsymbol{x}}_j^k,\hat{\boldsymbol{\lambda}}_j^{k})$, $j\neq i$, remain the same. That is 
\begin{align}
(\hat{\boldsymbol{x}}_j^{k+1},\hat{\boldsymbol{\lambda}}_j^{k+1})=(\hat{\boldsymbol{x}}_j^k,\hat{\boldsymbol{\lambda}}_j^{k})\quad \forall j\neq i.
\label{equ:PDMM_asyn}
\end{align}
The asynchronous updating scheme is more flexible  than the synchronous scheme in that no global node-coordination is needed for parameter-updating. 

In practice, the nodes in the graph can either be selected randomly or follow a predefined order for asynchronous parameter-updating. One scheme for realizing random node-activation is that after a node finishes parameter-updating, it randomly activates one of its neighbours for the next iteration. Another scheme is to introduce a clock at each node which ticks at the times of a (random) Poisson process (see \cite{Boyd06gossip} for detailed information). Each node is activated only when its clock ticks. As for node-activation in a predefined order, a cyclic updating scheme is most straightforward. Without loss of generality, the nodes can be selected in order from $i=1$ to $i=m=|\mathcal{V}|$ periodically, or equivalently $i_k=\textrm{mod}(i,m)$ where $\textrm{mod}(\cdot,\cdot)$ denotes the modulus operation.  Once node $i$ finishes its parameter updating operation, it informs node $i+1$ for next iteration. For the case that node $i$ and $i+1$ are not neighbours, the path from node $i$ to $i+1$ can be pre-stored at node $i$ to facilitate the process. 

Differently from synchronous PDMM, the theoretical analysis for asynchronous PDMM depends on how the individual node is selected per iteration. \cite{Zhang16PDMM} provides convergence analysis for the cyclic updating scheme while \cite{Sherson17PDMM} utlilizes monotone operator theory to conduct analysis for the random updating scheme.

%which we summarize in a theorem below:    
%\begin{theorem}
%Suppose at iteration $l$, node $i_l=\textrm{mod}(l,m)+1$ is selected for parameter-updating while the estimates of other nodes remain fixed as in (\ref{equ:PDMM_asyn}). If the estimate $\hat{\boldsymbol{x}}^{k}$ converges to a fixed point $\boldsymbol{x}'$ ($\lim_{k\rightarrow \infty}\hat{\boldsymbol{x}}^{k}=\boldsymbol{x}'$), we have $\boldsymbol{x}' =\boldsymbol{x}^{\star}$.
%\label{theorem:asyn}
%\end{theorem}

\subsection{Message-passing framework}
In this subsection, we replace each dual variable $\boldsymbol{\lambda}_{j|i}$ by a new variable $\boldsymbol{m}_{j\rightarrow i}$ in the updating procedure of PDMM.  We refer to an estimate $\hat{\boldsymbol{m}}_{j\rightarrow i}$ as a \emph{message} from node $j$ to node $i$. We will show in the following that when updating $\hat{\boldsymbol{x}}_{i}$ at node $i$, only the neighboring messages $\{\hat{\boldsymbol{m}}_{j\rightarrow i}|j\in \mathcal{N}_i\}$ are required instead of $\{(\hat{\boldsymbol{x}}_j,\hat{\boldsymbol{\lambda}}_{j|i})|j\in \mathcal{N}_i\}$ in the original updating expression. 

Our main motivation to replace $\{\boldsymbol{\lambda}_{j|i}\}$ with $\{\boldsymbol{m}_{j\rightarrow i}\}$ is that the introduction of $\{\boldsymbol{m}_{j\rightarrow i}\}$ makes it easier to analyze the relationship between PDMM and the Kalman filter in later sections. 
Further, the new updating expression involving $\{\boldsymbol{m}_{j\rightarrow i}\}$ is more similar to the update expressions of other types of message-passing methods, such as the linear-coordinate decent (LiCD) method \cite{xiaoqiang12LiCDQO} and the min-sum method \cite{Moallemi09GaBP,Moallemi10BPConvex}. %We will not discuss the similarities between those update expressions in this work as it is out of research focus.   

\begin{table}[t!]
\caption{Synchronous PDMM}
\label{table:PDMM} \vspace{-2mm}
\centering
\begin{tabular}{l}
  \hline
   \textrm{Initialization}: $\hat{\boldsymbol{x}}$ and $\{\hat{\boldsymbol{m}}_{i\rightarrow j}|i\in \mathcal{V}, j\in\mathcal{N}_i \}$\\
   Repeat\\
   \hspace{3mm} For all $i\in \mathcal{V}$ do \\
   \hspace{5mm}  $\hat{\boldsymbol{x}}_i^{k+1}\hspace{-0.5mm}=\hspace{-0.5mm}\arg\min_{\boldsymbol{x}_i}\hspace{-1mm}\Big[f_i(\boldsymbol{x}_i)\hspace{-0.5mm} $ \\
	\hspace{33mm}$+\underset{{j\in \mathcal{N}_i}}{\sum}\frac{1}{2}\| \boldsymbol{A}_{i | j}\boldsymbol{x}_i\hspace{-0.5mm} -\hat{\boldsymbol{m}}_{j\rightarrow i}^k\|_{\boldsymbol{P}_{ij}^{-1}}^2\Big]$ \\
	   \hspace{5mm} For all  $j\in \mathcal{N}_i$ do \\
   \hspace{7mm} $\hat{\boldsymbol{m}}_{i\rightarrow j}^{k+1}=\hat{\boldsymbol{m}}_{j\rightarrow i}^{k}+(\boldsymbol{c}_{ij}-2\boldsymbol{A}_{i | j}\hat{\boldsymbol{x}}_i^{k+1})$ \\
	 \hspace{5mm} End For \\
	 \hspace{3mm} End For \\
	 \hspace{3mm}$k\leftarrow k+1$ \\
	Until some stopping criterion is met \\
  \hline
\end{tabular}
\vspace{-5mm}
\end{table}

Formally, we define the variable $\boldsymbol{m}_{j\rightarrow i}$ to be
\begin{align}
\boldsymbol{m}_{j\rightarrow i} &= \boldsymbol{P}_{ij}{\boldsymbol{\lambda}}_{j|i} + (\boldsymbol{c}_{ij}-\boldsymbol{A}_{j | i}{\boldsymbol{x}}_j) \quad i\in\mathcal{V},  j\in \mathcal{N}_i, \label{equ:message_ji_def}
\end{align}
which is a function of ${\boldsymbol{\lambda}}_{j|i}$ and ${\boldsymbol{x}}_j$. 
By using algebra, one can show that the updating expression (\ref{equ:x_updateSyn0}) for $\hat{\boldsymbol{x}}_i^{k+1}$, $i\in \mathcal{V}$, can be simplified as
\begin{align}
\hspace{-3mm}\hat{\boldsymbol{x}}_i^{k+1}\hspace{-1mm}=&\arg\min_{\boldsymbol{x}_i}\hspace{-1mm}\Bigg[
\hspace{-0.5mm}\sum_{j\in \mathcal{N}_i}\hspace{-0.8mm}\frac{1}{2}\left\|\boldsymbol{A}_{i | j}\boldsymbol{x}_i\hspace{-0.6mm}-\hspace{-0.6mm}\hat{\boldsymbol{m}}_{j\rightarrow i}^k\right\|_{\boldsymbol{P}_{ij}^{-1}}^2 \hspace{-0.6mm}+\hspace{-0.6mm}f_i(\boldsymbol{x}_i)\hspace{-0.5mm}\Bigg], \label{equ:x_updateSyn1}
\end{align}
where 
\begin{align}
\hat{\boldsymbol{m}}_{j\rightarrow i}^k = \boldsymbol{P}_{ij}\hat{\boldsymbol{\lambda}}_{j|i}^k + (\boldsymbol{c}_{ij}-\boldsymbol{A}_{j | i}\hat{\boldsymbol{x}}_j^k)\quad j\in \mathcal{N}_i.\label{equ:message_ji_k_def}
\end{align}
Once $\hat{\boldsymbol{x}}_i^{k+1}$ is computed, the new message  $\hat{\boldsymbol{m}}_{i\rightarrow j}^{k+1}$ from  $i$ to $j$  can be expressed in terms of  $\hat{\boldsymbol{m}}_{j\rightarrow i}^k $  and $\hat{\boldsymbol{x}}_i^{k+1}$  as 
\begin{align}
\hspace{-2mm}\hat{\boldsymbol{m}}_{i\rightarrow j}^{k+1}\hspace{-0.5mm}=& 
\boldsymbol{P}_{ij}\hat{\boldsymbol{\lambda}}_{i|j}^{k+1}+(\boldsymbol{c}_{ij}-\boldsymbol{A}_{i | j}\hat{\boldsymbol{x}}_i^{k+1} ) \nonumber
 \\
=& \hat{\boldsymbol{m}}_{j\rightarrow i}^{k} +(\boldsymbol{c}_{ij}-2\boldsymbol{A}_{i | j}\hat{\boldsymbol{x}}_i^{k+1} ),\label{equ:message_update} 
\end{align}
where the expression is derived by using  (\ref{equ:lambda_updateSyn0}) and (\ref{equ:message_ji_k_def}).

Equ.~(\ref{equ:x_updateSyn1}) and (\ref{equ:message_update}) together specify the new form of updating procedure in terms of $\boldsymbol{x}$ and $\{\boldsymbol{m}_{i\rightarrow j}|  i\in \mathcal{V}, j\in \mathcal{N}_i\}$.  Table~\ref{table:PDMM} summarizes the reformulated updating procedure for synchronous PDMM. %The asynchronous PDMM has a similar updating procedure.  %One can also work out the updating expressions for asynchronous PDMM is .      %Again the new updating procedure holds for the general optimization problem (\ref{equ:problemMultiNode_re}) for both synchronous and asynchronous PDMM.  

\begin{remark}
Replacing $\{\boldsymbol{\lambda}_{j|i}\}$ with $\{\boldsymbol{m}_{j\rightarrow i}\}$ in the updating expressions of PDMM may also be beneficial in practice. For instance, every time after finishing local computations using synchronous PDMM, each pair $(i,j)\in \mathcal{E}$ of neighbouring nodes only need to exchange their messages $\hat{\boldsymbol{m}}_{i \rightarrow j}$ and  $\hat{\boldsymbol{m}}_{j\rightarrow i}$ instead of their primal and dual estimates $(\hat{\boldsymbol{x}}_i, \hat{\boldsymbol{\lambda}}_{i | j})$ and $(\hat{\boldsymbol{x}}_j, \hat{\boldsymbol{\lambda}}_{j | i})$, reducing the number of transmission parameters.    
\end{remark}

\vspace{-2mm}
\section{Finite-Time Convergence of PDMM for quadratic optimization over \\ a Tree-Structured Graph}
\vspace{-1mm}
\label{sec:finiteConvergence}

In this section, we consider using PDMM to solve the quadratic optimization (\ref{equ:quadOptimization}) over a tree-structured graph.  The  research goal is to find a particular setup of the matrix set $\mathcal{P}$ which leads to finite-time convergence of PDMM.  

\subsection{Parameter selection}
\label{subsec:param_selection}

The selection of the matrix set $\mathcal{P}$ is crucial to make PDMM converge in finite steps. In the following we explain how to construct the matrices in  $\mathcal{P}$ recursively by using the quadratic matrices $\{\boldsymbol{\Sigma}_i|i\in \mathcal{V}\}$ and $\{(\boldsymbol{A}_{i | j},\boldsymbol{A}_{j | i})|(i,j)\in \mathcal{E}\}$ in (\ref{equ:quadOptimization}).   

In order to construct the matrices in  $\mathcal{P}$, we first convert the (undirected) tree-structured graph $G=\{\mathcal{V},\mathcal{E}\}$ into a directed acyclic graph  $\vec{G}=\{\mathcal{V},\vec{\mathcal{E}}\}$. Firstly, we select any node $r\in \mathcal{V}$ to be a root node in the graph.  We then convert every undirected edge of $\mathcal{E}$ into a directed edge of $\vec{\mathcal{E}}$ w.r.t. the root node $r$. To do so, we define the distance between $i\in \mathcal{V}$ and $r$ as the minimum number of edges required for $i$ to reach $r$ by walking through a sequence of connected edges in the graph, denoted as $\mathrm{dist}(i,r)$. Since the graph is tree-structured, the shortest sequence of connected edges starting from $i$ till $r$ can be directly constructed by moving towards $r$ at each step.  Upon introducing the distance, $\vec{\mathcal{E}}$ can then be obtained as  
\begin{align}
\vec{\mathcal{E}} =\left\{[i,j]|| (i,j)\in \mathcal{E}, \mathrm{dist}(j,r)<\mathrm{dist}(i,r) \right\}.
\end{align}
 The set $\vec{\mathcal{E}}$ ensures that starting from any node $i$ and traveling over the directed graph $\vec{G}$ would eventually arrive at the root node $r$. Finally, we introduce 
 \begin{align}
 \mathrm{dist}(\vec{G},r) = \max_{i\in \mathcal{V}} \mathrm{dist}(i,r), \label{equ:dis_r}
 \end{align} 
which measures the maximum distance between any node  $i\in  \mathcal{V}$ and the root node $r$.  As will be shown in next subsection, the quantity $\mathrm{dist}(\vec{G},r)$ is related to how many iterations are required before PDMM converges to the optimal solution.  

We now explain in detail how to construct the matrices in $\mathcal{P}$ over $\vec{G}=(\mathcal{V},\vec{\mathcal{E}})$.  We start from the edges in $\vec{\mathcal{E}}$ with no preceding edges.  Suppose $[u,v]\in \vec{\mathcal{E}}$ has no preceding edges, implying that node $u$ has only one neighbour $v$ (i.e., $\mathcal{N}_u= \{v\}$). The matrix $\boldsymbol{P}_{uv}$ is then built as 
\begin{align}
\boldsymbol{P}_{uv}&=\boldsymbol{A}_{u | v}\boldsymbol{\Sigma}_u^{-1}\boldsymbol{A}_{u | v}^T,
\label{equ:optimalP_leafNode}
\end{align} 
where $\boldsymbol{\Sigma}_u$ is assumed to be nonsingular. With (\ref{equ:optimalP_leafNode}), the other matrices in $\mathcal{P}$ can then be constructed recursively along the directed edges in $\vec{\mathcal{E}}$ towards the root node $r$.  Suppose we are at the position to determine the matrix $\boldsymbol{P}_{ij}$, $[i,j]\in \vec{\mathcal{E}}$.  It can be computed by using the matrices $\{ \boldsymbol{P}_{ui} | u\in \mathcal{N}_{i}/j \}$ of preceding edges as (See Fig.~\ref{fig:prob_graph})
\begin{align}
\hspace{-4.5mm}\boldsymbol{P}_{ij}&\hspace{-0.6mm}=\hspace{-0.6mm}\boldsymbol{A}_{i | j}\hspace{-1mm}\left(\boldsymbol{\Sigma}_i\hspace{-0.6mm}+\hspace{-0.6mm} \sum_{u\in \mathcal{N}_{i}/ j} \hspace{-1mm} \boldsymbol{A}_{i | u}^T\boldsymbol{P}_{u i}^{-1}\boldsymbol{A}_{i | u}\right)^{-1}\hspace{-1mm}\boldsymbol{A}_{i | j}^T. \hspace{-2mm}
\label{equ:optimalParam}
\end{align}

We note that the above construction scheme (\ref{equ:optimalP_leafNode})-(\ref{equ:optimalParam}) implicitly assumes the constructed matrices in $\mathcal{P}$ are nonsingular, as required by (\ref{equ:x_updateSyn1}). %(\ref{equ:x_updateSyn0})-(\ref{equ:lambda_updateSyn0}).  
We summarize the assumption below:  

\begin{assumption}
Every matrix $\boldsymbol{P}_{ij}\in \mathcal{P}$ constructed by (\ref{equ:optimalP_leafNode})-(\ref{equ:optimalParam}) is assumed to be symmetric positive definite, i.e., $\boldsymbol{P}_{ij}\succ \boldsymbol{0}$.
\label{assumption:fullRowRank}
\end{assumption}

\begin{figure}[tb]
\centering
\begin{footnotesize}
  \includegraphics[width=60mm]{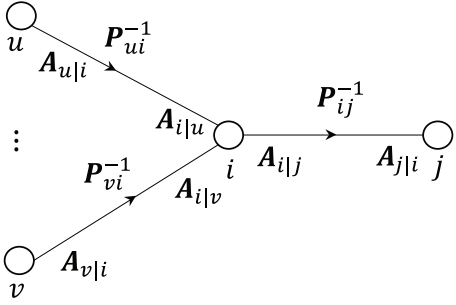}
\end{footnotesize}
\caption{\small  Recursive matrix construction over the directed graph $\vec{G}=\{\mathcal{V},\vec{\mathcal{E}}\}$. Each matrix $\boldsymbol{P}_{ij}$ has no direction, i.e., $\boldsymbol{P}_{ij}=\boldsymbol{P}_{ji}$. } \label{fig:prob_graph}
\vspace{-3.5mm}
\end{figure}

%The above condition ensures that every matrix $\boldsymbol{P}_{uv}$ in the set $\mathcal{P}$ is invertable. Further, we note that every equality constraint $\boldsymbol{A}_{i | j}\boldsymbol{x}_i+ \boldsymbol{A}_{j | i}\boldsymbol{x}_j=\boldsymbol{c}_{ij}$ has two matrices $\boldsymbol{A}_{i | j}$ and $\boldsymbol{A}_{j | i}$. The above assumption only requires one matrix to have full row-rank, which is determined by the direction of the edge $(i,j)$ in $\vec{\mathcal{E}}$. In other words, different root node $r$ would put the assumption on different subsets of $\{(\boldsymbol{A}_{j | i},\boldsymbol{A}_{i | j}) | (i,j)\in \mathcal{E}\}$ matrices. 

\vspace{-1mm}
\subsection{Convergence properties of PDMM}
\vspace{-1mm}
In this subsection we argue that the particular matrix selection (\ref{equ:optimalP_leafNode})-(\ref{equ:optimalParam})  leads to finite-time convergence of both synchronous and asynchronous PDMM irrespective of any initialization $\hat{\boldsymbol{x}}$ and $\{\hat{\boldsymbol{m}}_{i\rightarrow j}|, i\in \mathcal{V}, j\in \mathcal{N}_i \}$. 

In the following we first explain that with  (\ref{equ:optimalP_leafNode})-(\ref{equ:optimalParam}), information of the quadratic function $\{f_i\}$ and the equality constraints in (\ref{equ:quadOptimization}) flows towards the root node $r$ fast. The results will be summarized in a lemma below. We will then analyze the finite-time convergence of synchronous and asynchronous PDMM based on the lemma.  

\subsubsection{Effective information flow towards the root node $r$}

Consider the computation of the message $\hat{\boldsymbol{m}}_{i\rightarrow j}^{k+1}$ over a directed edge $[i,j]\in \vec{\mathcal{E}}$. It is clear from (\ref{equ:message_update}) that $\hat{\boldsymbol{m}}_{i\rightarrow j}^{k+1}$  is fully determined by the new estimate $\hat{\boldsymbol{x}}_{i}^{k+1}$ and the old message  $\hat{\boldsymbol{m}}_{j\rightarrow i}^{k}$. We note that in general, the computation of $\hat{\boldsymbol{x}}_i^{k+1}$ involves the messages $\{\hat{\boldsymbol{m}}_{u\rightarrow i}^k| u\in \mathcal{N}_i\}$ from all neighbouring nodes (see (\ref{equ:x_updateSyn1})). As a consequence, the message $\hat{\boldsymbol{m}}_{i\rightarrow j}^{k+1}$ implicitly makes use of the same messages from all neighboring nodes. We show in the following that with the particular matrix selection (\ref{equ:optimalP_leafNode})-(\ref{equ:optimalParam}), $\hat{\boldsymbol{m}}_{j\rightarrow i}^{k}$ is cancelled out in computing $\hat{\boldsymbol{m}}_{i\rightarrow j}^{k+1}$ for the considered quadratic optimization problem (\ref{equ:quadOptimization}), leading to effective information flow towards the root node $r$. 

\begin{lemma}
Let the set of matrices $\{\boldsymbol{P}_{ij}| (i,j)\in \mathcal{E}\}$ be constructed by following (\ref{equ:optimalP_leafNode})-(\ref{equ:optimalParam}) for the quadratic optimization problem (\ref{equ:quadOptimization}), where Assumption~\ref{assumption:fullRowRank} holds. By following the updating procedure (\ref{equ:x_updateSyn1}) and (\ref{equ:message_update}), the message $\hat{\boldsymbol{m}}_{i\rightarrow j}^{k+1}$ over the directed edge $[i, j]\in \vec{\mathcal{E}}$ is computed as
\begin{align}
\hspace{-2mm}\hat{\boldsymbol{m}}_{i\rightarrow j}^{k+1}\hspace{-0.5mm}
=& \boldsymbol{c}_{ij}-2\boldsymbol{A}_{i | j}  \Big(\boldsymbol{\Sigma}_i+\sum_{u\in \mathcal{N}_i}\boldsymbol{A}_{i | u}^T\boldsymbol{P}_{iu}^{-1}\boldsymbol{A}_{i | u}\Big)^{-1}  \nonumber \\
& \hspace{8mm}\cdot\Big(\boldsymbol{a}_i+\sum_{u\in \mathcal{N}_i/j}\boldsymbol{A}_{i | u}^T\boldsymbol{P}_{iu}^{-1}\hat{\boldsymbol{m}}_{u\rightarrow i}^k\Big).
\label{equ:mess_oneDirection}
\end{align}
\label{lemma:mess_oneDirection}\vspace{-3mm}
\end{lemma}
\begin{proof}
See Appendix~\ref{appendix:proof_lemma_mess_oneDirection} for the proof. 
\end{proof}

Lemma~\ref{lemma:mess_oneDirection} indicates that the computation of $\hat{\boldsymbol{m}}_{i\rightarrow j}^{k+1}$, $[i,j]\in \vec{\mathcal{E}}$, only involves the messages $\{\hat{\boldsymbol{m}}_{u\rightarrow i}^k| u\in \mathcal{N}_i/j \}$ from the preceding directed edges. From a high-level point of view, the message $\hat{\boldsymbol{m}}_{j\rightarrow i}^{k}$ is already available at node $j$ at time step $k$. Therefore, it seems redundant for the new message $\hat{\boldsymbol{m}}_{i\rightarrow j}^{k+1}$ to include any information of  $\hat{\boldsymbol{m}}_{j\rightarrow i}^{k}$ in computing $\hat{\boldsymbol{x}}_j^{k+2}$ at time step $k+1$. The particular form (\ref{equ:mess_oneDirection}) for $\hat{\boldsymbol{m}}_{i\rightarrow j}^{k+1}$ brings only new information to node $j$, which may accelerate the information flow over the graph. By running the iterates of either synchronous or asynchronous PDMM, global information of all the quadratic functions $\{f_i\}$ and the equality constraints in  (\ref{equ:quadOptimization}) would arrive at the root node $r$ in finite number of steps. As will be shown later the root node $r$ reaches optimality first, followed by other nodes. 

\subsubsection{Synchronous PDMM} We now consider synchronous PDMM where all the estimates $\hat{\boldsymbol{x}}$ and $\{ \hat{\boldsymbol{m}}_{i\rightarrow j} | i\in \mathcal{V}, j\in \mathcal{N}_i \}$ are updated simultaneously per iteration.  We have the following results: 

\begin{theorem}
Let the set of matrices $\{\boldsymbol{P}_{ij}| (i,j)\in \mathcal{E}\}$ be constructed by using (\ref{equ:optimalP_leafNode})-(\ref{equ:optimalParam}) for the quadratic optimization problem (\ref{equ:quadOptimization}), where Assumption~\ref{assumption:fullRowRank} holds. By following synchronous PDMM, it takes $\mathrm{dist}(\vec{G},r)+1$ iterations for the estimate $\hat{\boldsymbol{x}}_r$ to reach its optimal solution $\boldsymbol{x}_r^{\star}$.  The estimates $\{\hat{\boldsymbol{x}}_u,u\neq r\}$ of all the other nodes in $\mathcal{V}$ take at most $2\, \mathrm{dist}(\vec{G},r)+1$ iterations to arrive at their individual optimal solutions $\{{\boldsymbol{x}}_u^{\star},u\neq r\}$. 
\label{theorem:tree_synPDMM}
\end{theorem}

\begin{proof}
We first show that the estimate $\hat{\boldsymbol{x}}_r$ converges in $\mathrm{dist}(\vec{G},r)+1$ iterations. To do so,  we show that the messages $\{\hat{\boldsymbol{m}}_{i\rightarrow j} | [i,j]\in \vec{\mathcal{E}} \}$ over all the directed edges in $\vec{\mathcal{E}}$ converge to their fixed points in $\mathrm{dist}(\vec{G},r)$ iterations. From Lemma~\ref{lemma:mess_oneDirection}, it is clear that if a directed edge $[u,v]$ has no preceding directed edges (i.e., $\mathcal{N}_u=\{v\}$), the associated message $\hat{\boldsymbol{m}}_{u\rightarrow v}$ only needs one iteration to converge. Suppose $[u,v]$ has  preceding directed edges. In this situation, if the messages $\{\hat{\boldsymbol{m}}_{w\rightarrow u}| w\in \mathcal{N}_u /v \}$ of its preceding edges have already converged, the message $\hat{\boldsymbol{m}}_{u\rightarrow v}$ also needs only one iteration to converge (see (\ref{equ:mess_oneDirection})).  As $\vec{G}$ is an acyclic tree-structured graph, it is immediate from Theorem~\ref{theorem:syn} that the estimate $\hat{\boldsymbol{x}}_r^{k}=\boldsymbol{x}_r^{\star}$ when $k=\mathrm{dist}(\vec{G},r)+1$. 

In a similar manner, one can show that the estimates $\{\hat{\boldsymbol{x}}_u,u\neq r\}$ of other nodes converge in $2\, \mathrm{dist}(\vec{G},r)+1$ iterations. The key step is to show that the messages $\{\hat{\boldsymbol{m}}_{j\rightarrow i} | [i,j]\in \vec{\mathcal{E}} \}$ over the \emph{reverse direction} of all the directed edges in $\vec{\mathcal{E}}$ converge to their fixed points in $2 \,\mathrm{dist}(\vec{G},r)+1$ iterations by using (\ref{equ:x_updateSyn1}) and (\ref{equ:message_update}). We omit the argument as it is similar to the analysis above.  
\end{proof}

\begin{figure}[tb]
\centering
\begin{footnotesize}
  \includegraphics[width=70mm]{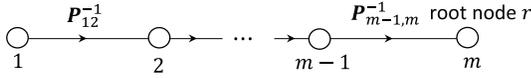}
\end{footnotesize}
\caption{\small  A directed chain graph with node $m$ being the root node $r$.  The set $\vec{\mathcal{E}}=\{[i,i+1] | i=1,\ldots, m-1\}$. } \label{fig:prob_graph_chain}
\end{figure}

\subsubsection{Asynchronous PDMM over a chain graph} \label{subsub:PDMM_Asyn_chain}
We now consider a chain graph for asynchronous PDMM  and postpone the analysis for a general tree-structured graph later on.  The analysis for a chain graph is relatively easy to understand and has a close relationship to the updating procedure of the Kalman filter (see Section~\ref{sec:RelationKalmanFilter}).   

\begin{table}[t!]
\caption{Asynchronous PDMM: Parameter updating over the chain graph in Fig.~\ref{fig:prob_graph_chain}. The notation $\boldsymbol{m}_{i\rightarrow j}^{\star}$ represents a fixed (or converged) message from node $i$ to $j$. }
\label{table:nodeSelChain} \vspace{-2mm}
\centering
\begin{tabular}{|l|}
\hline 
\hspace{0mm} Init.: Set $\{\boldsymbol{P}_{i,i+1}| (i,i+1)\in \mathcal{E}\}$ by using (\ref{equ:optimalP_leafNode})-(\ref{equ:optimalParam})   \\
  \hline
     \hspace{0mm} Forward computation: \\
   \hspace {3mm}   ${\boldsymbol{m}}_{1\rightarrow 2}^{\star}\hspace{-0.5mm}
\hspace{-0.7mm}=\hspace{-0.7mm}\boldsymbol{c}_{12}\hspace{-0.7mm}-\hspace{-0.7mm}2\boldsymbol{A}_{1 | 2}  \left(\boldsymbol{\Sigma}_1\hspace{-0.7mm}+\boldsymbol{A}_{1 | 2}^T\boldsymbol{P}_{12}^{-1}\boldsymbol{A}_{12}\right)^{-1}\boldsymbol{a}_1$  \\     
     \hspace{1mm} for node $i=2,\ldots, m-1$ do \\
   \hspace {3mm}   ${\boldsymbol{m}}_{i\rightarrow i+1}^{\star}\hspace{-0.5mm}
\hspace{-0.7mm}=\hspace{-0.7mm}\boldsymbol{c}_{i,i+1}\hspace{-0.7mm}-\hspace{-0.7mm}2\boldsymbol{A}_{i | i+1}  \Big(\boldsymbol{\Sigma}_i\hspace{-0.7mm}+\hspace{-0.7mm}\underset{{u\in \mathcal{N}_i}}{\sum}\hspace{-0.7mm}\boldsymbol{A}_{i | u}^T\boldsymbol{P}_{iu}^{-1}\boldsymbol{A}_{i | u}\Big)^{-1} $ \\
   $\hspace{20mm}\cdot\Big(\boldsymbol{a}_i+\boldsymbol{A}_{i | i-1}^T\boldsymbol{P}_{i,i-1}^{-1}{\boldsymbol{m}}_{i-1\rightarrow i}^{\star}\Big),$   $ [i,j]\in \vec{\mathcal{E}}$\\
   \hspace{1mm} end for \\
   \hline
     \hspace{0mm} Backward computation: \\   
    \hspace{1mm} for node $i=m,m-1,\ldots, 1$ do \\
   \hspace{3mm}  ${\boldsymbol{x}}_i^{\star} \hspace{-0.5mm}=\hspace{-0.5mm}\arg\underset{\boldsymbol{x}_i}{\min}\hspace{0mm}\Big[f_i(\boldsymbol{x}_i)\hspace{-0.5mm} +\underset{{j\in \mathcal{N}_i}}{\sum}\frac{1}{2}\| \boldsymbol{A}_{i | j}\boldsymbol{x}_i\hspace{-0.7mm} -{\boldsymbol{m}}_{j\rightarrow i}^{\star}\|_{\boldsymbol{P}_{ij}^{-1}}^2\Big]$ \\
   \hspace{3mm} ${\boldsymbol{m}}_{i\rightarrow i-1}^{\star}={\boldsymbol{m}}_{i-1\rightarrow i}^{\star}+(\boldsymbol{c}_{i,i-1}-2\boldsymbol{A}_{i | i-1}{\boldsymbol{x}}_i^{\star}),$  $i>1$ \\
	 \hspace{1mm} end for \\
  \hline
\end{tabular}
\vspace{-3.5mm}
\end{table}

 Before presenting the convergence results, we first describe how to select one node per iteration for parameter updating to maximize performance. To do so, we index the nodes of the chain graph from one side to the other side starting from $i=1$ until $i=m$ (see Fig.~\ref{fig:prob_graph_chain} for illustration). The last node $m$ is taken as the root node $r$.  Correspondingly, the directed edge set $\vec{\mathcal{E}}=\{[i,i+1]| i=1,\ldots, m-1\}$. The quantity $\mathrm{dist}(\vec{G},r)=m-1$. %There is only one path from node $1$ to the root node $r$ along the directed edges in  $\vec{\mathcal{E}}=\{[i,i+1]| i=1,\ldots, m-1\}$. 
 
 The nodes will be selected for parameter-updating in two stages. In the first stage, we select one node per iteration from $i=1$ to $i=m-1$  to perform a \emph{forward computation} towards the root node $r$. In the second stage, we select one node per iteration in reverse order from  $i=m$ to $i=1$ to perform a \emph{backward computation}. In total, the forward and backward computations account for $2m-1=2 \, \mathrm{dist}(\vec{G},r)+1$ iterations.  We have the following convergence results. 

\begin{theorem} Consider the quadratic optimization problem (\ref{equ:quadOptimization}) over the chain graph in Fig.~\ref{fig:prob_graph_chain}. Let the set of matrices $\{\boldsymbol{P}_{i,i+1}| (i,i+1)\in \mathcal{E}\}$ be constructed by using (\ref{equ:optimalP_leafNode})-(\ref{equ:optimalParam}), where Assumption~\ref{assumption:fullRowRank} holds. By following the updating procedure in Table~\ref{table:nodeSelChain},  the estimate $\hat{\boldsymbol{x}}_i$, $i=m,\ldots,1$, reaches its optimal solution $\boldsymbol{x}_i^{\star}$ after node $i$ finishes updating its parameters in the backward computation stage. %The other estimates   $\hat{\boldsymbol{x}}_i = \boldsymbol{x}_i^{}$ converges to their optimal solutions in at most $2 dist(\vec{G},r)+1$ iterations.   
\label{theorem:chain_asynPDMM}
\end{theorem}

\begin{proof}
The convergence analysis is similar to the proof for Theorem~\ref{theorem:tree_synPDMM}.  Firstly we conclude from Lemma~\ref{lemma:mess_oneDirection} that the forward computation in Table~\ref{table:nodeSelChain} leads to the fixed messages $\{\boldsymbol{m}_{i\rightarrow i+1}^{\star}| [i,i+1]\in \vec{\mathcal{E}} \}$ over all the directed edges in $\vec{\mathcal{E}}$.   

Next we consider the backward computation from node $i=m$ to $i=1$ in Table~\ref{table:nodeSelChain}. The computation at the root node $i=m$ would lead to the optimal estimate $\hat{\boldsymbol{x}}_m=\boldsymbol{x}_m^{\star}$ and fixed (or converged) message $\hat{\boldsymbol{m}}_{m\rightarrow m-1}={\boldsymbol{m}}_{m\rightarrow m-1}^{\star}$ because it uses the fixed message $\boldsymbol{m}_{m-1\rightarrow m}^{\star}$ from neighbour $m-1$. Similarly, the computation at other nodes $i=m-1,\ldots, 1$ would also lead to the optimal estimate $\hat{\boldsymbol{x}}_i=\boldsymbol{x}_i^{\star}$ and fixed message $\hat{\boldsymbol{m}}_{i\rightarrow i-1}={\boldsymbol{m}}_{i\rightarrow i-1}^{\star}$. The proof is complete.
\end{proof}

We note that for the updating scheme in Table~\ref{table:nodeSelChain}, the estimate $\hat{\boldsymbol{x}}_m$ of the root node reaches its optimal solution ${\boldsymbol{x}}_m^{\star}$  first, which is similar to the convergence results of  synchronous PDMM (See Theorem~\ref{theorem:tree_synPDMM}). If one is only interested in the optimal solution ${\boldsymbol{x}}_m^{\star}$, the backward computation in Table~\ref{table:nodeSelChain} can be greatly reduced by only performing computation at the root node $i=m$.  

 \subsubsection{Asynchronous PDMM over a tree-structured graph} Similarly to the case of a chain graph above, the key step is to properly select one node per iteration for parameter updating.  As the considered graph is more general, it may take more iterations than that for the chain graph to converge. 
 
The updating procedure for a tree-structured graph is presented in Table~\ref{table:nodeSelection}. Similarly to Table~\ref{table:nodeSelChain}, it also involves the forward and backward computations. When performing forward computation in Table~\ref{table:nodeSelection}, we select a directed edge $[i,j]\in \vec{\mathcal{E}}$ per iteration to compute the associated message $\boldsymbol{m}_{i\rightarrow j}^{\star}$, which is equivalent to the selection of node $i$. The convergence results are summarized below:

\begin{table}[t!]
\caption{Asynchronous PDMM: Parameter updating over a tree-structured graph. The notation $\boldsymbol{m}_{i\rightarrow j}^{\star}$ represents a fixed (or converged) message from node $i$ to $j$. }
\label{table:nodeSelection} \vspace{-2mm}
\centering
\begin{tabular}{|l|}
\hline 
\hspace{0mm} Init.: Set $\{\boldsymbol{P}_{i,j}| (i,j)\in \mathcal{E}\}$ by using (\ref{equ:optimalP_leafNode})-(\ref{equ:optimalParam})   \\
  \hline
     Forward computation: \\
     \hspace{1mm} Select $[i,j]\in \vec{\mathcal{E}}$ when $\{\hat{\boldsymbol{m}}_{u\rightarrow i} = {\boldsymbol{m}}_{u\rightarrow i}^{\star} | u \hspace{-0.6mm}\in \hspace{-0.6mm} \mathcal{N}_i/ j\}$ \\
   \hspace {5mm}    $\hspace{-2mm}{\boldsymbol{m}}_{i\rightarrow j}^{\star}\hspace{-0.5mm}
=\boldsymbol{c}_{ij}-2\boldsymbol{A}_{i | j}  \Big(\boldsymbol{\Sigma}_i+\underset{{u\in \mathcal{N}_i}}{\sum}\boldsymbol{A}_{i | u}^T\boldsymbol{P}_{iu}^{-1}\boldsymbol{A}_{i | u}\Big)^{-1} $ \\
   $\hspace{27mm}\cdot\Big(\boldsymbol{a}_i+\underset{{u\in \mathcal{N}_i/j}}{\sum}\boldsymbol{A}_{i | u}^T\boldsymbol{P}_{iu}^{-1}{\boldsymbol{m}}_{u\rightarrow i}^{\star}\Big)$  \\
 \hline  
   Backward Computation: \\
       \hspace{1mm} Select node $i$ when $\{\hat{\boldsymbol{m}}_{u\rightarrow i} = {\boldsymbol{m}}_{u\rightarrow i}^{\star} | u \hspace{-0.6mm}\in \hspace{-0.6mm} \mathcal{N}_i\}$ \\ %$\hat{\boldsymbol{m}}_{j\rightarrow i}={\boldsymbol{m}}_{j\rightarrow i}^{\star} $, $[i,j]\in \vec{\mathcal{E}}$  \\
   \hspace{2mm}  ${\boldsymbol{x}}_i^{\star} \hspace{-0.5mm}=\hspace{-0.5mm}\arg\underset{\boldsymbol{x}_i}{\min}\hspace{0mm}\Big[f_i(\boldsymbol{x}_i)\hspace{-0.5mm} +\underset{{j\in \mathcal{N}_i}}{\sum}\frac{1}{2}\| \boldsymbol{A}_{i | j}\boldsymbol{x}_i\hspace{-0.5mm} -{\boldsymbol{m}}_{j\rightarrow i}^{\star}\|_{\boldsymbol{P}_{ij}^{-1}}^2\Big]$ \\
   \hspace{2mm} ${\boldsymbol{m}}_{i\rightarrow u}^{\star}={\boldsymbol{m}}_{u\rightarrow i}^{\star}+(\boldsymbol{c}_{iu}-2\boldsymbol{A}_{i | u}{\boldsymbol{x}}_i^{\star}),$  $[u,i]\in \vec{\mathcal{E}}$ \\
  \hline
\end{tabular}
\vspace{-3.5mm}
\end{table}

\begin{theorem}
Let the set of matrices $\{\boldsymbol{P}_{ij}| (i,j)\in \mathcal{E}\}$ be constructed by using (\ref{equ:optimalP_leafNode})-(\ref{equ:optimalParam}) for the quadratic optimization problem (\ref{equ:quadOptimization}), where Assumption~\ref{assumption:fullRowRank} holds.  By following the updating procedure of  Table~\ref{table:nodeSelection}, the estimate $\hat{\boldsymbol{x}}$ converges to $\boldsymbol{x}^{\star}$ in finite iterations.   
\label{theorem:tree_asynPDMM}
\end{theorem}

\begin{proof}
The proof is similar to that for Theorem~\ref{theorem:chain_asynPDMM}. 
\end{proof}

%\begin{remark}
Besides the updating procedure of Table~\ref{table:nodeSelection},  one can also select the nodes in a different manner for parameter updating. For instance, the nodes can be selected either randomly or periodically in a predefined order. It is not difficult to show that the algorithm would still converge in finite iterations as long as every node is selected once in a while after reasonable iterations. However, it may take more iterations to converge than that of Table~\ref{table:nodeSelection}. 
%\end{remark}

\begin{remark}
The analysis in this section suggests that when using PDMM to minimize general decomposable convex functions, the matrix set $\mathcal{P}$ should be set based on the Hessian matrices of the objective functions if they exist for fast convergence speed. %We leave the optimal parameter selection of PDMM for general convex functions for future investigation.     
\end{remark}

 %Suppose a root node is chosen from the graph. We show that by recursive and proper parameter selection from all the leaf nodes towards the root node, PDMM converges in finite steps. Further the update at the root node reaches optimality first while those of the leaf nodes reach optimality last. We also study the relationships between PDMM and Kalman filter for the estimation problem over a linear statistical model. 
 
%Through experiment, we find that if the scalar parameters $\{{\gamma}_{ij}| (i,j)\in E \}$ are selected properly, PDMM converges in finite steps (at most $2|V|$ steps) regardless of any initialization of $({x},{\beta})$. To do so, one node is chosen from the graph and taken as the root note $r\in V$. The scalar parameters can then be  determined recursively starting from every leaf node towards the root node $r$. Suppose the parameter $\gamma_{ij}$ is to be determined, where node $j$ is closer to the root node $r$ than node $i$.  Then $\gamma_{ij}$ is computed as $\gamma_{ij}=1+\sum_{m\in N_i, m\neq j} \gamma_{mi}$. Empirical results show that the updates of PDMM at root node $r$  reaches optimality first, while those of the leaf nodes reach optimality last. The above preliminary results suggest that distributed algorithms motivated from the framework of convex optimization can have the same convergence speed as the inference methods motivated from machine learning.  

\section{Relation to Kalman Filter}
\label{sec:RelationKalmanFilter}

In this section we first briefly present the estimation problem over a statistical linear state-space model, where the notations of the state-space model are adopted from \cite{KailathBook} with slight modification. We will show that the estimation problem can be reformulated as a maximum log-likelihood (ML) problem.  We then describe the Kalman filter for solving the estimation problem. After that, we demonstrate how to apply the results of PDMM in Section~\ref{sec:finiteConvergence} to address the ML problem. We will show that the update expressions of PDMM and the Kalman filter are equivalent using algebra. 
\subsection{Estimation over a statistical linear state-space model}

Suppose the random processes $\{\boldsymbol{z}_l | l\geq 0\}$ and $\{\boldsymbol{y}_l | l \geq 0 \}$ follow a statistical linear state-space model 
\begin{align}
\hspace{-10mm} \boldsymbol{z}_{l+1} = \boldsymbol{F}_l\boldsymbol{z}_l + \boldsymbol{G}_l\boldsymbol{u}_l & \hspace{10mm} l \geq 0 \label{equ:stateSpace_z} \\
\hspace{-10mm} \boldsymbol{y}_l  = \boldsymbol{H}_l\boldsymbol{z}_l + \boldsymbol{v}_l & \hspace{10mm}  l \geq 0,
\label{equ:stateSpace_y}
\end{align}
where the matrices $\left\{\left[\boldsymbol{F}_l\textrm{ } \boldsymbol{G}_l\right] | l\geq 0\right\}$ are assumed to be of full row-rank, and the processes $\{\boldsymbol{v}_l | l\geq 0\}$ and $\{\boldsymbol{u}_l | l\geq 0\}$ are assumed to be $q \times 1$ and $r\times 1$ vector-valued zero-mean white Gaussian processes, with 
\begin{align}
\hspace{-3.5mm} E \left(\left[\begin{array}{c} \boldsymbol{u}_l \\ \boldsymbol{v}_l \end{array}\right] \left[\begin{array}{cc} \boldsymbol{u}_j^T & \boldsymbol{v}_j^T \end{array}\right]  \right) &\hspace{-1mm}= \hspace{-1mm}\left[\begin{array}{cc} \boldsymbol{Q}_l \delta_{lj} & \boldsymbol{0} \\ \boldsymbol{0} &\boldsymbol{R}_l\delta_{lj}  \end{array}\right]  \;\;\;   l\geq 0,\hspace{-2mm} \label{equ:cov}
\end{align}
where $E(\cdot)$ denotes expectation,  $\delta_{lj}=1$ when $l=j$, and $\delta_{lj}=0$ otherwise. We assume both $\boldsymbol{u}_l$ and $\boldsymbol{v}_l$ are non-degenerate Gaussian vectors (i.e., $\boldsymbol{Q}_l\succ 0$ and $\boldsymbol{R}_l\succ 0$). The initial state $\boldsymbol{z}_0$ has a Gaussian distribution with zero mean and covariance matrix  $\boldsymbol{\Pi}_0\succ0 $, which is independent of $\{\boldsymbol{u}_l | l\geq 0\}$ and $\{\boldsymbol{v}_l | l\geq 0 \}$.  The matrices $\boldsymbol{F}_l$ (of dimension $n\times n$), $\boldsymbol{G}_l (n\times r)$, $\boldsymbol{H}_l$ ($q\times n$), $\boldsymbol{Q}_l$ ($r\times r$), $\boldsymbol{R}_l$ ($q\times q$), and $\boldsymbol{\Pi}_0$ ($n\times n$) are known a priori.  

The process $\{\boldsymbol{z}_l | l\geq 0 \}$ is used to model the (hidden) state of a physical system over time while the process $\{\boldsymbol{y}_l | l\geq 0\}$ models the measurements over time. Correspondingly, the two processes $\{\boldsymbol{u}_l | l\geq 0\}$ and $\{\boldsymbol{v}_l | l\geq 0\}$ are often called \emph{process noise} and \emph{measurement noise}, respectively. The goal is how to estimate the hidden states efficiently and accurately from the measurements.       

\begin{remark}
One can also consider the general case that the process noise $\boldsymbol{u}_l$ is correlated with the measurement noise $\boldsymbol{v}_l$ for each $l\geq 0$. That is    
\begin{align}
\hspace{-3.5mm} E \left(\left[\begin{array}{c} \boldsymbol{u}_l \\ \boldsymbol{v}_l \end{array}\right] \left[\begin{array}{cc} \boldsymbol{u}_l^T & \boldsymbol{v}_l^T \end{array}\right]  \right) &\hspace{-1mm}= \hspace{-1mm}\left[\begin{array}{cc} \boldsymbol{Q}_l  & \boldsymbol{S}_l \\ \boldsymbol{S}_l^T &\boldsymbol{R}_l \end{array}\right]  \;\;\;   l\geq 0,\hspace{-2mm} \label{equ:cov_extend}
\end{align}
where $\boldsymbol{S}_l \neq \boldsymbol{0}$.  The analysis in the remainder of the paper can be extended the above general case using algebra.  
\end{remark}

Formally, the objective at time step $l$ is to compute the optimal estimate $\hat{\boldsymbol{z}}_{l+1| l}^{\star}$ for the state $\boldsymbol{z}_{l+1}$ in terms of the past observed measurements $\{\boldsymbol{y}_i | l \geq i \geq 0 \}$
 by minimizing the following mean squared error (MSE)
%\begin{align}
%\hat{\boldsymbol{x}}_{i+1} = \arg\min_{\hat{\boldsymbol{x}}_{i+1}}  E\left[\|\boldsymbol{x}_{i+1}-\hat{\boldsymbol{x}}_{i+1}\|^2\right]  
%\end{align}
\begin{align}
\hat{\boldsymbol{z}}_{l+1 | l}^{\star} &= \arg\min_{\hat{\boldsymbol{z}}_{l+1}}  E\left[\|\boldsymbol{z}_{l+1}-\hat{\boldsymbol{z}}_{l+1}\|^2 | \{\boldsymbol{y}_i | l \geq i \geq 0 \}\right]. \label{equ:z_MMSE}
\end{align}
It can be easily shown that the estimate $\hat{\boldsymbol{z}}_{l+1 | l}^{\star}$ takes the form of  (see Theorem 3.A.1 of \cite{KailathBook})
\begin{align}
\hat{\boldsymbol{z}}_{l+1 | l}^{\star}  =  E[\boldsymbol{z}_{l+1} | \{\boldsymbol{y}_i | l \geq i \geq 0 \}], \label{equ:z_condMean}
\end{align}
which is the conditional expectation of $\boldsymbol{z}_{l+1}$ given the measurements $\{\boldsymbol{y}_i | l \geq i \geq 0 \})$. 

Next we show that the above minimum mean squared error  (MMSE) criterion is equivalent to the maximum log-likelihood (ML) criterion in computing $\hat{\boldsymbol{z}}_{l+1 | l}^{\star}$. The results are summarized in a lemma below:

\begin{lemma}
The estimate $\hat{\boldsymbol{z}}_{l+1 | l}^{\star}$ in (\ref{equ:z_condMean}) can be alternatively obtained using the maximum log-likelihood (ML) criterion
 \begin{align}
 &\hat{\boldsymbol{z}}_{l+1 | l}^{\star} = \arg\max_{z_{l+1}} \Big[\max_{\{\boldsymbol{z}_{i} |  l \geq i \geq 0 \}} \ln p(\{ \boldsymbol{z}_i | l+1\geq i\geq 0 \}, \nonumber  \\
 & \hspace{50mm} \{ \boldsymbol{y}_i | l \geq i \geq 0 \})\Big], \label{equ:z_ML} 
 \end{align}
 where $p(\{ \boldsymbol{z}_i | l+1\geq i \geq 0 \}, 
 \{ \boldsymbol{y}_i | l \geq i \geq 0 \})$ represents the joint Gaussian distribution of the random vectors $\{ \boldsymbol{z}_i | l+1\geq i \geq 0 \}$ and $\{ \boldsymbol{y}_i | l \geq i \geq 0 \}$.
\label{lemma:ML_MMSE}
\end{lemma}
\begin{proof}
See Appendix~\ref{appendix:proof_lemma_ML_MMSE} for the proof.
\end{proof}

The equivalence of the MMSE and ML in computing $\hat{\boldsymbol{z}}_{l+1 | l}^{\star}$ above is due to the fact that all the random vectors in (\ref{equ:stateSpace_z})-(\ref{equ:stateSpace_y}) are Gaussian distributed. We will apply the asynchronous PDMM in Table~\ref{table:nodeSelChain} to address the ML problem (\ref{equ:z_ML}) later on. 

\begin{remark}
The assumption on the matrices $\left\{\left[\boldsymbol{F}_l\textrm{ } \boldsymbol{G}_l\right] | l\geq 0\right\}$ in (\ref{equ:stateSpace_z})-(\ref{equ:stateSpace_y}) being of full row-rank ensures that when applying PDMM to solve (\ref{equ:z_ML}), the constructed matrices in $\mathcal{P}$ are nonsingular. On the other hand, the above assumption is not neccessarily required for the Kalman filter to work. % satisfying Assumption~\ref{assumption:fullRowRank}.  
\end{remark}

%\begin{remark}
%In practice, the considered physical system might be nonlinear. In this case, a linear state-space model can be derived to approximate the physical system. The corresponding Kalman filter for solving (\ref{equ:z_MMSE})-(\ref{equ:z_condMean}) for the linearized version of the physical system is referred to as the extended Kalman filter \cite{Ribeiro04EKF}.   
%\end{remark}

\subsection{Kalman filter}
\label{subsec:Kalman}

By making use of the linear structure of the state-space model (\ref{equ:stateSpace_z})-(\ref{equ:stateSpace_y}), the Kalman filter is able to compute the estimates $\{\hat{\boldsymbol{z}}_{l+1 | l }^{\star}| l\geq 0\}$ recursively and efficiently over time. In the following we briefly describes the recursive update expressions of the Kalman filter (see Chapter~9 of \cite{KailathBook}). 

Before presenting the update expressions for $\{\hat{\boldsymbol{z}}_{l+1 | l }^{\star} | l \geq 0\}$, we first need to introduce some notations. We denote the covariance matrix of the error vector $\tilde{\boldsymbol{z}}_{l | l-1}=\boldsymbol{z}_l - \hat{\boldsymbol{z}}_{l |l-1}^{\star}$ as  $\boldsymbol{D}_l$, i.e., $\boldsymbol{D}_l = E[\tilde{\boldsymbol{z}}_{l | l-1} \tilde{\boldsymbol{z}}_{l | l-1}^T]$.  When $l=0$, we let  
\begin{align}
\boldsymbol{D}_0 = E[\tilde{\boldsymbol{z}}_{0| -1}   \tilde{\boldsymbol{z}}_{0| -1}^T] = E[{\boldsymbol{z}}_{0}   {\boldsymbol{z}}_{0}^T] = \boldsymbol{\Pi}_0.  \label{equ:D_init_Kalman}
\end{align}
Equivalently we can set the initial estimate $\hat{\boldsymbol{z}}_{0 | -1}^{\star} $ to 
\begin{align} 
\hat{\boldsymbol{z}}_{0 | -1}^{\star} = \boldsymbol{0}.  \label{equ:z_init_Kalman}
\end{align}
The covariance matrix $\boldsymbol{D}_l$ characterizes the uncertainty of the estimate $\hat{\boldsymbol{z}}_{l | l-1 }^{\star}$ for the state $\boldsymbol{z}_l$.  A small $\boldsymbol{D}_l$ implies that $\hat{\boldsymbol{z}}_{l | l-1 }^{\star}$ is a good estimate with small uncertainty.  

Suppose the estimate $\hat{\boldsymbol{z}}_{l|l-1}^{\star}$ and the associated covariance matrix $\boldsymbol{D}_l$ for the state $\boldsymbol{z}_l$ were already obtained at time step $l-1$.  At time step $l$, the Kalman filter computes $\hat{\boldsymbol{z}}_{l+1 | l}^{\star} $ in terms of $\hat{\boldsymbol{z}}_{l | l-1}^{\star}$ and $\boldsymbol{y}_l$ as
\begin{align}
\hat{\boldsymbol{z}}_{l+1|l}^{\star} = (\boldsymbol{F}_l-\boldsymbol{K}_{l}\boldsymbol{H}_l)\hat{\boldsymbol{z}}_{l | l-1}^{\star} + \boldsymbol{K}_{l} \boldsymbol{y}_l  \;\;\;  l\geq 0, \label{equ:z_update_Kalman}
\end{align}
where the optimal Kalman gain $\boldsymbol{K}_{l}$ is computed by using $\boldsymbol{D}_l$ as 
\begin{align}
\boldsymbol{K}_{l} &= \boldsymbol{F}_l\boldsymbol{D}_l\boldsymbol{H}_l^T(\boldsymbol{H}_l \boldsymbol{D}_l \boldsymbol{H}_l^T+\boldsymbol{R}_l)^{-1} \;\;\;  l\geq 0 . \label{equ:K_update_Kalman}
\end{align}
By using $\boldsymbol{K}_{l}$ and $\boldsymbol{D}_l$, the associated covariance matrix $\boldsymbol{D}_{l+1}$ for $\hat{\boldsymbol{z}}_{l+1 | l}^{\star}$ can be computed as 
\begin{align}
\boldsymbol{D}_{l+1} &= \boldsymbol{F}_l\boldsymbol{D}_l \boldsymbol{F}_l^T +\boldsymbol{G}_l \boldsymbol{Q}_l \boldsymbol{G}_l^T -\boldsymbol{K}_l(\boldsymbol{H}_l\boldsymbol{D}_l \boldsymbol{H}_l^T+\boldsymbol{R}_l)\boldsymbol{K}_l \nonumber\\ & \hspace{56mm}\;\;\;  l\geq 0  . \label{equ:D_update_Kalman}
\end{align} 
Equ.~(\ref{equ:D_init_Kalman})-(\ref{equ:D_update_Kalman}) together specify the recursive update expressions of the Kalman filter. 

It is not difficult to see from (\ref{equ:z_update_Kalman}) that $\hat{\boldsymbol{z}}_{l+1 | l}^{\star}$ is in fact a linear combination of historical measurements $\{\boldsymbol{y}_i | l \geq i \geq 0 \}$. The contributions of $\boldsymbol{y}_i$ from $i=0$ till $i=l-1$ to $\hat{\boldsymbol{z}}_{l+1 | l}^{\star}$ are fully embedded in $\hat{\boldsymbol{z}}_{l | l-1}^{\star}$. Therefore, the Kalman filter only needs to store the most recent estimate and its associated covariance matrix for the computation at next time step.

Conceptually speaking, the Kalman filter can be viewed as performing message-passing over a chain graph $G_c=(\mathcal{V}_c,\mathcal{E}_c)$ with $\mathcal{V}_c\hspace{-0.5mm}=\hspace{-0.5mm}\{ l | l\hspace{-0.5mm}=\hspace{-0.5mm}0,1,2, \ldots\}$ and $\mathcal{E}_c\hspace{-0.5mm}=\hspace{-0.5mm}\{(l,l+1)| l\hspace{-0.5mm}=\hspace{-0.5mm}0,1,2,\ldots\}$.
%\begin{align}\hspace{-2mm}\mathcal{V}\hspace{-0.5mm}=\hspace{-0.5mm}\{ i | i\hspace{-0.5mm}=\hspace{-0.5mm}0,1,2, \ldots\} \;  \textrm{and} \; \mathcal{E}\hspace{-0.5mm}=\hspace{-0.5mm}\{(i,i+1)| i\hspace{-0.5mm}=\hspace{-0.5mm}0,1,2,\ldots\}. \hspace{-1mm}\label{equ:chainGraph_Kalman}
%\end{align}
 Each time step $l$ of the state-space model (\ref{equ:stateSpace_z})-(\ref{equ:stateSpace_y}) corresponds to a node $l \in \mathcal{V}_c$ in the chain graph. When node $l$ receives the message $(\hat{\boldsymbol{z}}_{l | l-1}^{\star}, \boldsymbol{D}_{l})$ from node $l-1$, it computes and forwards the message $(\hat{\boldsymbol{z}}_{l+1|l}^{\star}, \boldsymbol{D}_{l+1})$ to the next node $l+1$. In other words, the Kalman filter only performs the forward computation starting from node $l=0$ to nodes of higher indices.  

Finally it is worthy noting that the computation (\ref{equ:D_update_Kalman}) of $\{\boldsymbol{D}_l | l\geq 0\}$ is independent of the estimates $\{\hat{\boldsymbol{z}}_{l+1 | l }^{\star} | l\geq 0\}$.  Instead, $\{\boldsymbol{D}_l | l \geq 0\}$ are constructed recursively based on the structure (i.e., $\{(\boldsymbol{F}_l,\boldsymbol{H}_l)| l\geq 0\}$) of the state-space model and the covariance matrices of the processes $\{(\boldsymbol{u}_l,\boldsymbol{v}_l)| l\geq 0\}$ and $\boldsymbol{z}_0$ over time.  The above property is similar to the computation (\ref{equ:optimalP_leafNode})-(\ref{equ:optimalParam}) of the matrix set $\mathcal{P}$ of PDMM in Section~\ref{sec:finiteConvergence}. We will show in next subsection that $\{\boldsymbol{D}_l | l \geq 0\}$ are equivalent to the matrix set $\mathcal{P}$ of PDMM obtained by using (\ref{equ:optimalP_leafNode})-(\ref{equ:optimalParam}) in solving the ML problem (\ref{equ:z_ML}).   

 \begin{remark}
 We note that the Kalman filter works even for singular matrices $\{\boldsymbol{D}_l | l\geq 0\}$ while PDMM requires all the matrices in $\mathcal{P}$ to be nonsingular. We conjecture that at least for a subclass (e.g., the ML problem (\ref{equ:z_ML})) of the optimization problem (\ref{equ:problemMultiNode_re}), the matrix set $\mathcal{P}$ might not need to be nonsingular for PDMM to work.     
 %That is the matrices $\{\left[ \boldsymbol{F}_l\textrm{ }\boldsymbol{G}_l \right] | l\geq 0 ]\}$ are not required to be of full row-rank for the Kalman filter to work. As will be shown in next subsection, the above conditions are imposed to allow the applications of PDMM.   
 \end{remark}

\subsection{Addressing the ML problem by using PDMM } 
In this subsection, we consider using the results of PDMM in Section~\ref{sec:finiteConvergence} to solve the ML problem (\ref{equ:z_ML}). To be able to draw connections between PDMM and the Kalman filter, we focus on the asynchronous PDMM (See the description in \ref{subsub:PDMM_Asyn_chain}).   

\subsubsection{Problem reformulation onto a chain graph} We note from Lemma~\ref{lemma:ML_MMSE} that $p(\{\boldsymbol{z}_i | l+1 \geq i \geq 0 \},\{\boldsymbol{y}_i |  l \geq i\geq 0\})$ represents a joint Gaussian distribution. The covariance matrix of the distribution over  $\{\boldsymbol{z}_i | l+1 \geq i \geq 0 \}$ and $\{\boldsymbol{y}_i |  l \geq i\geq 0\}$ is sparse due to the linear structure of the state-space model (\ref{equ:stateSpace_z})-(\ref{equ:stateSpace_y}) and the property (\ref{equ:cov}). As a result, taking the logarithm of the distribution produces a summation of quadratic functions     
\begin{align}
\hspace{-0mm}&\ln p( \{\boldsymbol{z}_i | l+1 \geq i\geq 0 \},\{\boldsymbol{y}_i |  l \geq i \geq 0\} )  \nonumber \\
\hspace{-0mm}&= -\sum_{i=0}^l g_i(\boldsymbol{z}_i, \boldsymbol{u}_i) + \mathrm{constant} \label{equ:sum_quad_state_space}
\end{align}
where the individual quadratic functions $g_i(\boldsymbol{z}_i, \boldsymbol{u}_i)$, $i=0,\ldots, l$, are 
expressed as
\begin{align}
g_0(\boldsymbol{z}_0,\boldsymbol{u}_0) &= \frac{1}{2}\hspace{-0.7mm}\left[\hspace{-1.5mm}\begin{array}{c}  \boldsymbol{u}_0 \\ \boldsymbol{y}_0 \hspace{-0.3mm}-\hspace{-0.3mm} \boldsymbol{H}_0 \boldsymbol{z}_0 \end{array} \hspace{-1.5mm} \right]^T\hspace{-0.8mm}
\left[\hspace{-1.5mm} \begin{array}{cc} \boldsymbol{Q}_0 &  \boldsymbol{0} \\ \boldsymbol{0}  & \boldsymbol{R}_0 \end{array} \hspace{-1.5mm}\right]^{-1} \hspace{-0.8mm}
\left[\hspace{-1.5mm} \begin{array}{c}  \boldsymbol{u}_0  \\ \boldsymbol{y}_0 - \boldsymbol{H}_0 \boldsymbol{z}_0 \end{array}\hspace{-1.5mm} \right] \nonumber\\
& +\frac{1}{2}\boldsymbol{z}_0 \boldsymbol{\Pi}_0^{-1} \boldsymbol{z}_0^T   \label{equ:sum_quad_state_space1}\\
g_i(\boldsymbol{z}_i,\boldsymbol{u}_i) &= \frac{1}{2}\hspace{-0.7mm}\left[\hspace{-1.5mm}\begin{array}{c} \boldsymbol{u}_i  \\ \boldsymbol{y}_i  \hspace{-0.3mm}-\hspace{-0.3mm} \boldsymbol{H}_i \boldsymbol{z}_i  \end{array} \hspace{-1.5mm} \right]^T\hspace{-0.8mm}
\left[\hspace{-1.5mm} \begin{array}{cc} \boldsymbol{Q}_i &  \boldsymbol{0} \\ \boldsymbol{0}^T  & \boldsymbol{R}_i \end{array} \hspace{-1.5mm}\right]^{-1} \hspace{-0.8mm}
\left[\hspace{-1.5mm} \begin{array}{c} \boldsymbol{u}_i \\  \boldsymbol{y}_i \hspace{-0.3mm}-\hspace{-0.3mm} \boldsymbol{H}_i \boldsymbol{z}_i \end{array}\hspace{-1.5mm} \right] \nonumber \\ & \hspace{40mm}   i=1,\dots, l  \label{equ:sum_quad_state_space2} 
\end{align}
where the set of vectors $(\boldsymbol{z}_i, \boldsymbol{u}_i)$, $i=0,\ldots,l+1$, satisfy   
\begin{align}
\boldsymbol{z}_{i+1} =& \boldsymbol{F}_i\boldsymbol{z}_i +\boldsymbol{G}_i \boldsymbol{u}_i \quad   i= 0,1,\ldots, l, \label{equ:sum_quad_state_space4}
\end{align}
which are obtained directly from (\ref{equ:stateSpace_z}).  %It is seen that the function on the right hand side of (\ref{equ:sum_quad_state_space}) is decomposable, which is due to the property (\ref{equ:cov}). %the set of vectors $\{\boldsymbol{u}_i | l \geq i\geq 0 \}$. %. to allow the reformulated problem (\ref{equ:quad_Kalman})-(\ref{equ:quad_Kalman1}) to have the same form as (\ref{equ:quadOptimization}). %The derivation of (\ref{equ:sum_quad_state_space})-(\ref{equ:sum_quad_state_space3}) above utilizes the structure of the state-space model (\ref{equ:stateSpace_z})-(\ref{equ:stateSpace_y}) and the property that $\{(\boldsymbol{u}_i, \boldsymbol{v}_i) | l \geq i \geq 0 \}$ are zero-mean white Gaussian processes and are independent of the initial Gaussian vector $\boldsymbol{z}_0$. 

Upon deriving (\ref{equ:sum_quad_state_space})-(\ref{equ:sum_quad_state_space4}), the ML problem (\ref{equ:z_ML}) can be reformulated as a decomposable quadratic optimization over a chain graph $G_{c,l+1}=(\mathcal{V}_{c,l+1},\mathcal{E}_{c,l+1})$ with $\mathcal{V}_{c,l+1}=\{i | i=0,1,\ldots, l+1\}$ and $\mathcal{E}_{c,l+1}=\{(i,i+1) | i=0,\ldots, l \}$
\begin{align}
&\hat{\boldsymbol{z}}_{l+1 | l}^{\star}=\arg\min_{z_{l+1}} \left[ \min_{ \{\boldsymbol{z}_i, \boldsymbol{u}_i\} } \sum_{i =0}^l g_i(\boldsymbol{z}_i, \boldsymbol{u}_i) \right]  \label{equ:quad_Kalman} \\
& \textrm{ s.t.  }  \left[\hspace{-0.7mm}\begin{array}{cc}  \boldsymbol{0} & \boldsymbol{I} \end{array} \hspace{-0.7mm} \right] \left[\hspace{-0.7mm}\begin{array}{c}\boldsymbol{u}_{i+1} \\ \boldsymbol{z}_{i+1} \end{array}\hspace{-0.7mm}\right]  \hspace{-0.7mm}+\hspace{-0.7mm} \left[\hspace{-0.7mm}\begin{array}{cc}  -\boldsymbol{G}_i & -\boldsymbol{F}_i \end{array} \right] \left[\begin{array}{c}\boldsymbol{u}_i \\ \boldsymbol{z}_i \end{array}\hspace{-0.7mm}\right] = \boldsymbol{0}\nonumber\\
&\hspace{50mm}  (i,i+1)\in \mathcal{E}_{c,l+1},\label{equ:quad_Kalman1}
\end{align}
which falls within the problem formulation of (\ref{equ:quadOptimization}).  We note that there is no local quadratic function at node $l+1$ in (\ref{equ:quad_Kalman}), which is because $\boldsymbol{y}_{l+1}$ is not available yet. 
%where 
%\begin{align}
% h_{l+1}(\boldsymbol{z}_{l+1} &,\boldsymbol{u}_{l+1}) =  0, \label{equ:fun_last_kalman}
%\end{align}
%where $g_l(\boldsymbol{z}_l, \boldsymbol{u}_l)$, $l=0,\ldots, i$, and $h_{i+1}(\boldsymbol{z}_{i+1})$ are given by (\ref{equ:sum_quad_state_space1})-(\ref{equ:sum_quad_state_space3}).    

It is clear that as the time step $l$ of the state-space model increases by 1 to be $l+1$, the above chain graph $G_{c,l+1}=(\mathcal{V}_{c,l+1},\mathcal{E}_{c,l+1})$ is extended to be $G_{c,l+2}=(\mathcal{V}_{c,l+2},\mathcal{E}_{c,l+2})$ with $\mathcal{V}_{c,l+2} = \mathcal{V}_{c,l+1}\cup \{l+2\}$ and $\mathcal{E}_{c,l+2} = \mathcal{E}_{c,l+1}\cup \{(l+1,l+2)\}$. We will show later that PDMM is able to compute $\hat{\boldsymbol{z}}_{l+2|l+1}^{\star}$ based on $\hat{\boldsymbol{z}}_{l+1|l}^{\star}$ and $\boldsymbol{y}_{l+1}$ without using the previous measurements $\{\boldsymbol{y}_i | l\geq i\geq 0 \}$. 

\subsubsection{Updating procedure of asynchronous PDMM} We now  consider applying the asynchronous PDMM presented in Table~\ref{table:nodeSelChain} to solve (\ref{equ:quad_Kalman})-(\ref{equ:quad_Kalman1}). To do so, we let $\boldsymbol{x}_i =[\boldsymbol{u}_i^T, \boldsymbol{z}_i^T]^T$, $i=1,\ldots,l+1$. Furthermore, we take the node $l+1$ in $\mathcal{V}_{c,l+1}$ to be the root node $r$ (see Fig.~\ref{fig:prob_graph_chain}).  With the root node $r=l+1$, the corresponding directed graph can be built for $G_{c,l+1}$, which we denote as $\vec{G}_{c,l+1} = (\mathcal{V}_{c,l+1},\vec{\mathcal{E}}_{c,l+1})$ where $\vec{\mathcal{E}}_{c,l+1}=\{[i,i+1] | i=0,\ldots, l\}$. 

The problem formulation (\ref{equ:quad_Kalman})-(\ref{equ:quad_Kalman1}) suggests that only the estimate $\hat{\boldsymbol{z}}_{l+1 | l}^{\star}$ at the root node $r=l+1$ is of interest. Therefore, the updating procedure in Table~\ref{table:nodeSelChain} can be simplified to the forward computation only from $i=0$ to $i=l$. The backward computation can be removed completely.  

To apply the simplified updating procedure in Table~\ref{table:nodeSelChain}, we first need to construct the matrix set $\mathcal{P}=\{\boldsymbol{P}_{i,i+1} | (i,i+1)\in \mathcal{E}_{c,l+1} \}$ by following the instructions presented in \ref{subsec:param_selection}.  We start with $\boldsymbol{P}_{01}$ as its corresponding directed edge $[0,1]$ has no preceding edges in $\vec{\mathcal{E}}_{c,l+1}$ (i.e., $\mathcal{N}_0 = \{1\}$). Combining (\ref{equ:quadOptimization}), (\ref{equ:optimalP_leafNode}), (\ref{equ:sum_quad_state_space1}), and (\ref{equ:quad_Kalman1}) produces 
\begin{align}
&\hspace{-2mm}\boldsymbol{P}_{01} = \left[\begin{array}{cc} \boldsymbol{G}_0 &  \boldsymbol{F}_0  \end{array}\right] \boldsymbol{J}_0^{-1}
\left[\begin{array}{c} \boldsymbol{G}_0 \\ \boldsymbol{F}_0^T \end{array}\right], \label{equ:P_SS_init}
\end{align}
where the matrix $\boldsymbol{J}_0$ is given by 
\begin{align}
\boldsymbol{J}_0=\left[\begin{array}{cc} \boldsymbol{Q}_0^{-1} & \boldsymbol{0} \\ \boldsymbol{0} & \boldsymbol{H}_0^T\boldsymbol{R}_0^{-1}\boldsymbol{H}_0 + \boldsymbol{\Pi}_{0}^{-1} \end{array}\right].  \label{equ:J_SS_init}
\end{align}
The matrix $\boldsymbol{J}_0$ is nonsingular as $\boldsymbol{\Pi}_0\succ \boldsymbol{0}$ and $\boldsymbol{Q}_0\succ \boldsymbol{0}$.  
It can be easily shown that $\boldsymbol{P}_{01}$ is also nonsingular as  $\left[\boldsymbol{F}_0\textrm{ } \boldsymbol{G}_0\right] $ is of full row-rank from (\ref{equ:stateSpace_z})-(\ref{equ:stateSpace_y}). 
With $\boldsymbol{P}_{01}$, the other matrices in $\mathcal{P}$ can be constructed recursively by combining (\ref{equ:quadOptimization}), (\ref{equ:optimalParam}), (\ref{equ:sum_quad_state_space2}), and (\ref{equ:quad_Kalman1}), which can be expressed as  
\begin{align}
&\hspace{-3.5mm}\boldsymbol{P}_{i,i+1} = \left[\begin{array}{cc}  \boldsymbol{G}_i & \boldsymbol{F}_i \end{array}\right]  \boldsymbol{J}_i^{-1}\left[\begin{array}{c} \boldsymbol{G}_i \\  \boldsymbol{F}_i^T  \end{array}\right] \;\;\; i=1,\ldots, l, \label{equ:P_SS}
\end{align}
where the matrix $\boldsymbol{J}_i$ is a function of $\boldsymbol{P}_{i-1,i}$, given by
\begin{align}
\boldsymbol{J}_i=\left[\begin{array}{cc} \boldsymbol{Q}_i^{-1} & \boldsymbol{0} \\ \boldsymbol{0} & \boldsymbol{H}_i^T\boldsymbol{R}_i^{-1}\boldsymbol{H}_i + \boldsymbol{P}_{i-1,i}^{-1} \end{array}\right]. \label{equ:J_SS}
\end{align} 
Similarly, the matrix $\boldsymbol{P}_{i,i+1} $ is also nonsingular due to the property that $\left[ \boldsymbol{F}_i\textrm{ } \boldsymbol{G}_i\right] $ is of full row-rank. Therefore,  Assumption \ref{assumption:fullRowRank} holds for the matrix set $\mathcal{P}$ constructed by (\ref{equ:P_SS_init})-(\ref{equ:J_SS}).

The relation between the matrix set $\mathcal{P}=\{\boldsymbol{P}_{i,i+1} | (i,i+1)\in\mathcal{E}_{c,l+1}\}$ constructed by (\ref{equ:P_SS_init})-(\ref{equ:J_SS}) and the set of covariance matrices $\{\boldsymbol{D}_i | i\geq 1\}$ of the Kalman filter is characterized in a proposition below:
\begin{proposition}
Let the matrix set $\mathcal{P}=\{\boldsymbol{P}_{i,i+1} | (i,i+1)\in\mathcal{E}_{c,l+1}\}$ be constructed by using  (\ref{equ:P_SS_init})-(\ref{equ:J_SS}) for the quadratic optimization (\ref{equ:quad_Kalman})-(\ref{equ:quad_Kalman1}). The matrices in $\mathcal{P}$ are related to the matrices $\boldsymbol{D}_i$, $i\geq 1$, of the Kalman filter as  
\begin{align}
\boldsymbol{D}_{i+1} = \boldsymbol{P}_{i,i+1}\quad i=0,\ldots, l.
\label{equ:D_P_equ}
\end{align}
\label{prop:D_P_relation}
\vspace{-5mm}
\end{proposition}
\begin{proof}
See Appendix~\ref{appendix:proof_D_P_relation} for the proof. 
\end{proof}

%Before using the updating scheme of Table~\ref{table:nodeSelChain}, we first need to check if Assumption \ref{assumption:fullRowRank} holds for the matrix set $\mathcal{P}$ constructed by (\ref{equ:P_SS_init})-(\ref{equ:J_SS}).  We first consider $\boldsymbol{P}_{01}$. By using the property that $\boldsymbol{\Pi}_0$ is nonsingular and $\left[-\boldsymbol{F}_0\textrm{ } -\boldsymbol{G}_0\right] $ is of full row-rank (see (\ref{equ:stateSpace_z})-(\ref{equ:stateSpace_y})), we conclude $\boldsymbol{P}_{01}$ is nonsingular. The non-singularity of  other matrices in $\mathcal{P}$ can be argued in a similar manner, which verifies Assumption \ref{assumption:fullRowRank}. 

We are now ready to consider the forward computation in Table~\ref{table:nodeSelChain}. We show in a theorem below that the message $\boldsymbol{m}_{i\rightarrow i+1}^{\star}$ is in fact the optimal solution $\hat{\boldsymbol{z}}_{i+1 | i}^{\star}$ for each $i=0,\ldots, l$. To simplify the analysis, we introduce the initial message $\boldsymbol{m}_{-1 \rightarrow 0}^{\star} = \boldsymbol{0}$ (equal to $\hat{\boldsymbol{z}}_{0|-1}^{\star}$ from  (\ref{equ:z_init_Kalman})) and the initial matrix $\boldsymbol{P}_{-1,0}= \boldsymbol{\Pi}_0 = \boldsymbol{D}_0$. 
\begin{theorem}
Let the matrix set $\mathcal{P}=\{\boldsymbol{P}_{i,i+1} | (i,i+1)\in\mathcal{E}_{c,l+1}\}$ be constructed by using  (\ref{equ:P_SS_init})-(\ref{equ:J_SS}) for the quadratic optimization (\ref{equ:quad_Kalman})-(\ref{equ:quad_Kalman1}). Then the messages $\{\boldsymbol{m}_{{i\rightarrow i+1}}^{\star} | i=0,\ldots, l\}$ generated by the forward computation of Table~\ref{table:nodeSelChain} can be computed recursively as
\begin{align}
\boldsymbol{m}_{i\rightarrow i+1}^{\star} &= (\boldsymbol{F}_i-\boldsymbol{K}_{i}\boldsymbol{H}_i)\hat{\boldsymbol{m}}_{i | i-1}^{\star} + \boldsymbol{K}_{i} \boldsymbol{y}_i \label{equ:message_recursive} \\
&=\hat{\boldsymbol{z}}_{i+1 | i}^{\star}\quad i=0,\ldots, l. 
\end{align}
\label{theorem:PDMM_Kalman_equ}
\end{theorem}
\begin{proof}
See Appendix~\ref{appendix:proof_PDMM_Kalman_equ} for the proof.
\end{proof}

Proposition~\ref{prop:D_P_relation} and Theorem~\ref{theorem:PDMM_Kalman_equ} together show that the asynchronous PDMM and the Kalman filter have the same update expressions. From the theory of the Kalman filter, the matrix $\boldsymbol{P}_{i,i+1}$, $i\geq 0 $, now has a physical meaning. That is  $\boldsymbol{P}_{i,i+1}$ reflects the uncertainty of the message (or equivalently, the estimator) $\boldsymbol{m}_{i\rightarrow i+1}^{\star}$ for the random variable $\boldsymbol{z}_{i+1}$. 

Consider using both the forward and backward computation of Table~\ref{table:nodeSelChain} for solving the quadratic optimization (\ref{equ:quad_Kalman})-(\ref{equ:quad_Kalman1}). In this case, the asynchronous PDMM can be interpreted as a \emph{ smoother} of the hidden states from the measurements. Furthermore, if the backward computation is implemented for only a fixed number $\kappa$ of steps   (i.e.., computation from $i=l$ to $i=l-\kappa$ for any $l\geq 0$),  the asynchronous PDMM can be interpreted as a \emph{fixed-lag smoother} (see  Chapter~4 of \cite{KailathBook}). 
 
\section{Conclusions} 
\label{sec:conclusion}
In this paper, we have optimized the convergence speed of PDMM for solving decomposable quadratic optimizations over tree-structured graphs. We provide a scheme on how to select the optimal matrix set $\mathcal{P}$ which leads to finite-time convergence of both synchronous and asynchronous PDMM. As an example, we apply asynchronous PDMM to perform causal estimation over a statistical  linear state-space model, where the estimation problem can be reformulated as a quadratic optimization over a  growing chain graph. We have shown that asynchronous PDMM and the Kalman filter share the same updating expressions under mild conditions. 

One future research direction is to  study for what problems, the nonsingularity conditions of the matrix set $\mathcal{P}$ can be relaxed or removed for PDMM. It is also of great interest to extend PDMM for direct estimation over nonlinear state-space models. Finally, it would be of great value to consider the optimal parameter selection of PDMM for general decomposable convex functions, which are more common than quadratic functions in practice.     

%It is worth noting that PDMM can handle other decomposable convex functions in addition to the quadratic functions,  of which the optimal parameter setup remains to be discovered. 

\appendices

\section{Proof for Lemma~\ref{lemma:mess_oneDirection}}
\label{appendix:proof_lemma_mess_oneDirection}

Before formally presenting the proof, we first introduce the Woodbury matrix identity in a lemma below:
\begin{lemma} [Woodbury matrix identity]
Suppose $\boldsymbol{A}\in \mathbb{R}^{m\times m}$ and $\boldsymbol{C}\in \mathbb{R}^{n\times n}$ are two nonsingular matrices.  For any matrix $\boldsymbol{U}\in \mathbb{R}^{m\times n}$ and $\boldsymbol{U}\in \mathbb{R}^{n\times m}$, there is 
\begin{align}
\hspace{-4mm}(\boldsymbol{A}\hspace{-0.6mm}+\hspace{-0.6mm}\boldsymbol{U}\hspace{-0.3mm}\boldsymbol{C}\hspace{-0.3mm}\boldsymbol{V})^{-1} \hspace{-0.6mm}\equiv\hspace{-0.6mm} \boldsymbol{A}^{-1} \hspace{-0.6mm}-\hspace{-0.6mm} \boldsymbol{A}^{-1}\hspace{-0.3mm}\boldsymbol{U}(\boldsymbol{C}^{-1}\hspace{-0.6mm}+\hspace{-0.6mm}\boldsymbol{V}\boldsymbol{A}^{-1}\hspace{-0.3mm}\boldsymbol{U})^{-1}\hspace{-0.3mm}\boldsymbol{V}\hspace{-0.3mm}\boldsymbol{A}^{-1}.\nonumber
\end{align}
\label{lemma:identity}\vspace{-3.5mm}
\end{lemma}

We now describe the proof for Lemma~\ref{lemma:mess_oneDirection}. The message $\hat{\boldsymbol{m}}_{i\rightarrow j}^{k+1}$ over the directed edge $[i,j]\in \vec{\mathcal{E}}$ can be computed by combining (\ref{equ:quadOptimization}),  (\ref{equ:x_updateSyn1}),  (\ref{equ:message_update}), (\ref{equ:optimalParam}) and Lemma~\ref{lemma:identity} in three steps.  Firstly, we compute $\hat{\boldsymbol{x}}_{i+1}^{k+1}$ by inserting the expression for $f_i(\boldsymbol{x}_i)$ in (\ref{equ:quadOptimization})  into (\ref{equ:x_updateSyn1}),  which can be expressed as
\begin{align}
\hspace{-0.2mm}\hat{\boldsymbol{x}}_i^{k+1}\hspace{-0mm}=& \Big(\hspace{-0mm}\boldsymbol{\Sigma}_i\hspace{-0.7mm}+\hspace{-0.8mm}\sum_{u\in \mathcal{N}_i}\hspace{-1.2mm}\boldsymbol{A}_{i | u}^T\boldsymbol{P}_{iu}^{-1}\boldsymbol{A}_{i | u}\hspace{-0.3mm}\Big)^{-1} \nonumber \\
&\hspace{-0.8mm}\cdot \Big(\hspace{-0.3mm}\boldsymbol{a}_i\hspace{-0.7mm}+\hspace{-0.9mm}\sum_{u\in \mathcal{N}_i}\hspace{-0.8mm}\boldsymbol{A}_{i | u}^T\boldsymbol{P}_{iu}^{-1}\hat{\boldsymbol{m}}_{u\rightarrow i}^k\hspace{-0.3mm}\Big). \nonumber 
\end{align}
Secondly, inserting the above expression for $\hat{\boldsymbol{x}}_{i}^{k+1}$ into (\ref{equ:message_update}) produces  
\begin{align}
\hspace{-2mm}\hat{\boldsymbol{m}}_{i\rightarrow j}^{k+1}
=& \boldsymbol{c}_{ij}-2\boldsymbol{A}_{i | j}  \Big(\boldsymbol{\Sigma}_i+\sum_{u\in \mathcal{N}_i}\boldsymbol{A}_{i | u}^T\boldsymbol{P}_{iu}^{-1}\boldsymbol{A}_{i | u}\Big)^{-1}  \nonumber 
\\
& \hspace{2mm}\cdot\Big(\boldsymbol{a}_i+\sum_{u\in \mathcal{N}_i/j}\boldsymbol{A}_{i | u}^T\boldsymbol{P}_{iu}^{-1}\hat{\boldsymbol{m}}_{u\rightarrow i}^k\Big) + \tilde{\boldsymbol{m}}_{j\rightarrow i}^k, \label{equ:message_combined}
\end{align}
where the quantity $\tilde{\boldsymbol{m}}_{j\rightarrow i}^k$ is given by  
\begin{align}
\tilde{\boldsymbol{m}}_{j\rightarrow i}^k=\boldsymbol{B}_{i\rightarrow j}\hat{\boldsymbol{m}}_{j\rightarrow i}^k, \label{equ:message_combined2}
\end{align}
where the matrix $\boldsymbol{B}_{i\rightarrow j} $ is expressed as
\begin{align}
\boldsymbol{B}_{i\rightarrow j} &\hspace{-0.7mm} = \hspace{-0.7mm} \Big[ \boldsymbol{I} \hspace{-0.7mm}-\hspace{-0.7mm} 2\boldsymbol{A}_{i | j}  \Big(\boldsymbol{\Sigma}_i \hspace{-0.7mm}+\hspace{-0.7mm} \sum_{u\in \mathcal{N}_i}\boldsymbol{A}_{i | u}^T\boldsymbol{P}_{iu}^{-1}\boldsymbol{A}_{i | u}\Big)^{-1} \boldsymbol{A}_{i | j}^T\boldsymbol{P}_{ij}^{-1}\Big]. \nonumber 
\end{align}
We note that the only difference between (\ref{equ:mess_oneDirection}) and (\ref{equ:message_combined}) is the last term  $\tilde{\boldsymbol{m}}_{j\rightarrow i}^k$.

In the final step, we show that the matrix $\boldsymbol{B}_{i\rightarrow j}$ before $\hat{\boldsymbol{m}}_{j\rightarrow i}^k$ in (\ref{equ:message_combined2}) is in fact zero by using (\ref{equ:optimalParam}) and Lemma~\ref{lemma:identity}, indicating that $\hat{\boldsymbol{m}}_{j\rightarrow i}^k$ has no contribution to the new message $\hat{\boldsymbol{m}}_{i\rightarrow j}^{k+1}$.  To make the derivation below more readable, we let 
\begin{align}
\boldsymbol{C}_{i/j} = \boldsymbol{\Sigma}_i+\sum_{u\in \mathcal{N}_i/j}\boldsymbol{A}_{i | u}^T\boldsymbol{P}_{iu}^{-1}\boldsymbol{A}_{i | u}\nonumber.
\end{align}
As a result, the matrix $\boldsymbol{P}_{ij}$ can be alternatively expressed in terms of $\boldsymbol{C}_{i/j} $ as $\boldsymbol{P}_{ij}=\boldsymbol{A}_{i | j}  \boldsymbol{C}_{i/j}^{-1}\boldsymbol{A}_{i | j}^T$ by using (\ref{equ:optimalParam}).
The expression for $\boldsymbol{B}_{i\rightarrow j}$ can then be simplified as
\begin{align}
&{\boldsymbol{B}}_{i\rightarrow j} \hspace{-0.7mm} =\hspace{-0.7mm}   \boldsymbol{I} \hspace{-0.7mm}-\hspace{-0.7mm} 2\boldsymbol{A}_{i | j}  \Big( \boldsymbol{C}_{i/j} \hspace{-0.7mm}+\hspace{-0.7mm}\boldsymbol{A}_{i | j}^T\boldsymbol{P}_{ij}^{-1}\boldsymbol{A}_{i | j} \Big)^{-1}\hspace{-0.7mm} \boldsymbol{A}_{i | j}^T\boldsymbol{P}_{ij}^{-1} \nonumber \\
& \stackrel{(a)}{=}  \boldsymbol{I} -2\boldsymbol{A}_{i | j}  \Big( \boldsymbol{C}_{i/j}^{-1}-\boldsymbol{C}_{i/j}^{-1}\boldsymbol{A}_{i | j}^T(\boldsymbol{P}_{ij}+\boldsymbol{A}_{i | j}\boldsymbol{C}_{i/j}^{-1}\boldsymbol{A}_{i | j}^T)^{-1} \nonumber \\
&\hspace{32mm}\cdot\boldsymbol{A}_{i | j} \boldsymbol{C}_{i/j}^{-1}\Big) \boldsymbol{A}_{i | j}^T\boldsymbol{P}_{ij}^{-1} \nonumber \\
& \stackrel{(b)}{=}  \boldsymbol{I} -2\boldsymbol{A}_{i | j}  \Big( \boldsymbol{C}_{i/j}^{-1}-\frac{1}{2}\boldsymbol{C}_{i/j}^{-1}\boldsymbol{A}_{i | j}^T\boldsymbol{P}_{ij}^{-1}\boldsymbol{A}_{i | j} \cdot \boldsymbol{C}_{i/j}^{-1}\Big) \boldsymbol{A}_{i | j}^T\boldsymbol{P}_{ij}^{-1} \nonumber \\
& \stackrel{(c)}{=} \boldsymbol{I} - 2 \Big( \boldsymbol{P}_{ij} - \frac{1}{2}\boldsymbol{P}_{ij} \Big) \boldsymbol{P}_{ij}^{-1}\nonumber \\
& =\boldsymbol{0},
\end{align}
where step~$(a)$ makes use of Lemma~\ref{lemma:identity}, and step~$(b)$ and $(c)$ use the expression $\boldsymbol{P}_{ij}=\boldsymbol{A}_{i | j}  \boldsymbol{C}_{i/j}^{-1}\boldsymbol{A}_{i | j}^T$. The proof is complete. 

\section{Proof for Lemma~\ref{lemma:ML_MMSE}}
\label{appendix:proof_lemma_ML_MMSE}

Before presenting the proof, we first describe a property of a Gaussian marginal distribution in a lemma below: 
\begin{lemma}
Assume a random vector  $\boldsymbol{z}=[\boldsymbol{z}_1^T,\boldsymbol{z}_2^T]^T$ has a Gaussian distribution with mean vector $\boldsymbol{\mu}$ and covariance matrix $\boldsymbol{\Sigma}$, expressed as
\begin{align}
\boldsymbol{\mu} = \left[\begin{array}{c} \boldsymbol{\mu}_1 \\   \boldsymbol{\mu}_2 \end{array}\right] \; \textnormal{and} \; 
\boldsymbol{\Sigma} = \left[\begin{array}{cc} \boldsymbol{\Sigma}_{11} & \boldsymbol{\Sigma}_{12}  \\ \boldsymbol{\Sigma}_{12}^T   & \boldsymbol{\Sigma}_{22} \end{array}\right]. \nonumber
\end{align} 
Then the component vector $\boldsymbol{z}_i$, $i=1,2$, is also Gaussian distributed with mean vector $\boldsymbol{\mu}_i$ and covariance matrix $\boldsymbol{\Sigma}_{ii}$. 
\label{lemma:GaussianMarginal}
\end{lemma}

We now provide the derivation for (\ref{equ:z_ML}). The estimate $\hat{\boldsymbol{z}}_{l+1 | l}^{\star}$ in  (\ref{equ:z_condMean})  can be computed alternatively as 
\begin{align}
\hat{\boldsymbol{z}}_{l+1 | l}^{\star}  &=  E[\boldsymbol{z}_{l+1}  | \{\boldsymbol{y}_i | l \geq i \geq 0 \}]  \nonumber \\
 &\stackrel{(a)}{=} \arg\max_{\boldsymbol{z}_{l+1}} \ln p(\boldsymbol{z}_{l+1} |  \{\boldsymbol{y}_i | l \geq i \geq 0 \}) \nonumber \\
 & \stackrel{(b)}{=} \arg\max_{z_{l+1}} \Big[ \max_{ \{\boldsymbol{z}_{i} |  l\geq i\geq 0 \} }  \nonumber \\
 & \hspace{17mm} \ln p(\{ \boldsymbol{z}_i | l+1\geq i \geq 0  \} |  \{\boldsymbol{y}_i | l \geq i \geq 0 \})  \Big] \nonumber \\
 & = \arg\max_{z_{l+1}}\Big[ \max_{ \{\boldsymbol{z}_{i} |  l \geq i\geq 0 \} } \nonumber \\
 & \hspace{17mm} \ln p(\{ \boldsymbol{z}_i | l+1 \geq i \geq 0\}, \{ \boldsymbol{y}_i | l \geq i \geq 0 \}) \Big], \nonumber 
\end{align}
where step $(a)$ uses the fact that the conditional probability density function $p(\boldsymbol{z}_{l+1} |  \{\boldsymbol{y}_i | l \geq i \geq 0 \})$ is a Gaussian distribution, and step~(b) uses Lemma~\ref{lemma:GaussianMarginal} and the fact that  $p(\{ \boldsymbol{z}_i | l+1\geq i \geq 0  \} |  \{\boldsymbol{y}_i | l \geq i \geq 0 \}$ is a Gaussian distribution.  The proof is complete.

\section{Proof for Proposition~\ref{prop:D_P_relation}}
\label{appendix:proof_D_P_relation}

We first introduce $\boldsymbol{P}_{-1,0}= \boldsymbol{\Pi}_0$ to unify the expressions (\ref{equ:J_SS_init}) and (\ref{equ:J_SS}) for $\boldsymbol{J}_0$ and $\boldsymbol{J}_i$, $i \geq 1$.  The matrix $\boldsymbol{P}_{i,i+1}$, $i=0,\ldots, l$, can then be rewritten as
\begin{align}
\hspace{0mm}\boldsymbol{P}_{i,i+1} =&\boldsymbol{G}_i \boldsymbol{Q}_i \boldsymbol{G}_i^T  + \boldsymbol{F}_i  ( \boldsymbol{P}_{i-1,i}^{-1} + \boldsymbol{H}_i^T\boldsymbol{R}_i^{-1}\boldsymbol{H}_i  )^{-1}  \boldsymbol{F}_i^T  \nonumber \\
\stackrel{(a)}{=}& \boldsymbol{G}_i \boldsymbol{Q}_i \boldsymbol{G}_i^T  + \boldsymbol{F}_i \boldsymbol{P}_{i-1,i}\boldsymbol{F}_i^T-\boldsymbol{F}_i  \boldsymbol{P}_{i-1,i} \boldsymbol{H}_i^T\nonumber\\
&\hspace{15mm}\cdot (\boldsymbol{R}_i +\boldsymbol{H}_i \boldsymbol{P}_{i-1,i}\boldsymbol{H}_i^T)^{-1}\boldsymbol{H}_i \boldsymbol{P}_{i-1,i} \boldsymbol{F}_i^T  \nonumber\\
\stackrel{(b)}{=}& \boldsymbol{G}_i \boldsymbol{Q}_i\boldsymbol{G}_i^T + \boldsymbol{F}_i \boldsymbol{P}_{i-1,i}\boldsymbol{F}_i^T \nonumber\\
&- \boldsymbol{K}_{i}  (\boldsymbol{R}_i +\boldsymbol{H}_i \boldsymbol{P}_{i-1,i}\boldsymbol{H}_i^T) \boldsymbol{K}_{i}^T \nonumber \\
\stackrel{(c)}{=}& \boldsymbol{D}_{i+1},  \nonumber
\end{align}
where step $(a)$ uses Lemma~\ref{lemma:identity}, step $(b)$ uses (\ref{equ:K_update_Kalman}), and step $(c)$ uses (\ref{equ:D_update_Kalman}). The proof is complete. 

\section{Proof for Theorem~\ref{theorem:PDMM_Kalman_equ}}
\label{appendix:proof_PDMM_Kalman_equ}

By combining (\ref{equ:quadOptimization}), (\ref{equ:sum_quad_state_space2}), (\ref{equ:quad_Kalman1}), and (\ref{equ:J_SS}) with 
$\boldsymbol{x}_i \hspace{-0.3mm}=\hspace{-0.3mm} [\boldsymbol{u}_i^T,\boldsymbol{z}_i^T]^T$, 
the expression for $\boldsymbol{m}_{i\rightarrow i+1}^{\star}$ in Table~\ref{table:nodeSelChain} can be rewritten as 
\begin{align}
 &\hspace{-3.5mm} {\boldsymbol{m}}_{i\rightarrow i+1}^{\star}\hspace{-0.5mm} \nonumber \\
%\hspace{-0.7mm}=&\boldsymbol{c}_{i,i+1}\hspace{-0.7mm}-\hspace{-0.7mm}2\boldsymbol{A}_{i | i+1}  \Big(\boldsymbol{\Sigma}_i\hspace{-0.7mm}+\hspace{-0.7mm}\underset{{u\in \mathcal{N}_i}}{\sum}\hspace{-0.7mm}\boldsymbol{A}_{i | u}^T\boldsymbol{P}_{iu}^{-1}\boldsymbol{A}_{i | u}\Big)^{-1}  \nonumber \\
%   &\hspace{20mm}\cdot\Big(\boldsymbol{a}_i+\boldsymbol{A}_{i | i-1}^T\boldsymbol{P}_{i,i-1}^{-1}{\boldsymbol{m}}_{i-1\rightarrow i}^{\star}\Big), \nonumber \\
   =& 2 [\begin{array}{cc} \boldsymbol{G}_i  & \boldsymbol{F}_i \end{array}] \left(\boldsymbol{J}_i+ \left[\begin{array}{c} \boldsymbol{G}_i^T  \\ \boldsymbol{F}_i^T \end{array}\right]\boldsymbol{P}_{i,i+1}^{-1} [\begin{array}{cc} \boldsymbol{G}_i  & \boldsymbol{F}_i \end{array}] \right)^{-1}\nonumber \\
   &\hspace{20mm}\cdot \left[\begin{array}{c} \boldsymbol{0} \\ \boldsymbol{P}_{i-1,i}^{-1}\boldsymbol{m}_{i-1 \rightarrow i}^{\star}+ \boldsymbol{H}_i^T\boldsymbol{R}_i^{-1}\boldsymbol{y}_i   \end{array}\right] \nonumber \\
 %=& 2 [\begin{array}{cc} \boldsymbol{G}_i  & \boldsymbol{F}_i \end{array}] \left(\boldsymbol{J}_i^{-1} -  \frac{1}{2}\boldsymbol{J}_i^{-1}  \left[\begin{array}{c} \boldsymbol{G}_i^T  \\ \boldsymbol{F}_i \end{array}\right] \boldsymbol{P}_{i,i+1}^{-1}  \left[\begin{array}{cc} \boldsymbol{G}_i & \boldsymbol{F}_i \end{array}\right]\boldsymbol{J}_i^{-1} \right)\nonumber   \\
 %&  \hspace{20mm}\cdot \left[\begin{array}{c} \boldsymbol{0} \\ \boldsymbol{H}_i^T\boldsymbol{R}_i^{-1}\boldsymbol{y}_i +\boldsymbol{P}_{i-1,i}^{-1}\boldsymbol{m}_{i-1 \rightarrow i}^{\star}  \end{array}\right] \nonumber \\
 \stackrel{(a)}{=}&[\begin{array}{cc} \boldsymbol{G}_i  & \boldsymbol{F}_i \end{array}] \boldsymbol{J}_i^{-1}\left[\begin{array}{c} \boldsymbol{0} \\  \boldsymbol{P}_{i-1,i}^{-1}\boldsymbol{m}_{i-1 \rightarrow i}^{\star} +\boldsymbol{H}_i^T\boldsymbol{R}_i^{-1}\boldsymbol{y}_i  \end{array}\right] \nonumber 
 \end{align}
 \begin{align}
& \hspace{-15.5mm}\stackrel{(b)}{=}\boldsymbol{F}_i\left( \boldsymbol{P}_{i-1,i}^{-1}+\boldsymbol{H}_i^T\boldsymbol{R}_i^{-1}\boldsymbol{H}_i \right)^{-1}\nonumber\\
 &\hspace{0mm}\cdot\left( \boldsymbol{P}_{i-1,i}^{-1}\boldsymbol{m}_{i-1 \rightarrow i}^{\star} +\boldsymbol{H}_i^T\boldsymbol{R}_i^{-1}\boldsymbol{y}_i \right), \label{equ:message_recursive1}
\end{align}
where step $(a)$ uses Lemma~\ref{lemma:identity} and (\ref{equ:P_SS}), and step $(b)$ uses (\ref{equ:J_SS}). 

Next we show that the two matrices multiplying $\boldsymbol{m}_{i-1 \rightarrow i}^{\star} $ and $\boldsymbol{y}_i$ in (\ref{equ:message_recursive1}) are the same as those in (\ref{equ:message_recursive}).  We first consider the matrix multiplying $\boldsymbol{m}_{i-1 \rightarrow i}^{\star}$, which can be simplified as 
\begin{align}
 &\boldsymbol{F}_i\left(\boldsymbol{P}_{i-1,i}^{-1} + \boldsymbol{H}_i^T\boldsymbol{R}_i^{-1}\boldsymbol{H}_i\right)^{-1}\boldsymbol{P}_{i-1,i}^{-1} \nonumber \\
 &\stackrel{(a)}{=} \boldsymbol{F}_i \left[\boldsymbol{I} -  \boldsymbol{P}_{i-1,i}\boldsymbol{H}_i^T(\boldsymbol{R}_i +\boldsymbol{H}_i \boldsymbol{P}_{i-1,i}\boldsymbol{H}_i^T )^{-1}\boldsymbol{H}_i \right] \nonumber \\
 &\stackrel{(b)}{=}  \boldsymbol{F}_i  -  \boldsymbol{K}_{i} \boldsymbol{H}_i,   \label{equ:matrix_equivalence1} 
\end{align}
where step $(a)$ uses Lemma~\ref{lemma:identity}, and step $(b)$ uses (\ref{equ:K_update_Kalman}) and Proposition~\ref{prop:D_P_relation}. 

The matrix multiplying $\boldsymbol{y}_i$ in (\ref{equ:message_recursive1}) can also be simplified as 
\begin{align}
 &\boldsymbol{F}_i\left(\boldsymbol{P}_{i-1,i}^{-1} + \boldsymbol{H}_i^T\boldsymbol{R}_i^{-1}\boldsymbol{H}_i\right)^{-1}\boldsymbol{H}_i^T\boldsymbol{R}_i^{-1} \nonumber \\
 &= \boldsymbol{F}_i  \Big[\boldsymbol{P}_{i-1,i} -  \boldsymbol{P}_{i-1,i}\boldsymbol{H}_i^T(\boldsymbol{R}_i +\boldsymbol{H}_i \boldsymbol{P}_{i-1,i}\boldsymbol{H}_i^T )^{-1}\boldsymbol{H}_i  \nonumber \\
&\hspace{50mm} \cdot \boldsymbol{P}_{i-1,i} \Big]  \boldsymbol{H}_i^T\boldsymbol{R}_i^{-1} \nonumber \\
 &= \boldsymbol{F}_i \boldsymbol{P}_{i-1,i}\boldsymbol{H}_i^T \Big[\boldsymbol{R}_i^{-1}   -  (\boldsymbol{R}_i +\boldsymbol{H}_i \boldsymbol{P}_{i-1,i}\boldsymbol{H}_i^T )^{-1} \nonumber\\
 &\hspace{30mm}\cdot \left(\boldsymbol{H}_i  \boldsymbol{P}_{i-1,i} \boldsymbol{H}_i^T+\boldsymbol{R}_i-\boldsymbol{R}_i\right)\boldsymbol{R}_i^{-1}  \Big]  \nonumber  \\
 &= \boldsymbol{F}_i  \boldsymbol{P}_{i-1,i}\boldsymbol{H}_i^T(\boldsymbol{R}_i +\boldsymbol{H}_i \boldsymbol{P}_{i-1,i}\boldsymbol{H}_i^T )^{-1} \nonumber \\
 &\stackrel{(a)}{=} \boldsymbol{K}_{i},   \label{equ:matrix_equivalence2} 
\end{align}
where the last step $(a)$ uses (\ref{equ:K_update_Kalman}) and Proposition~\ref{prop:D_P_relation}. Combining (\ref{equ:message_recursive1})-(\ref{equ:matrix_equivalence2}) produces (\ref{equ:message_recursive}). The proof is complete. 

\ifCLASSOPTIONcaptionsoff
  \newpage
\fi

\bibliographystyle{IEEEtran}
\bibliography{sigProcessing}

% Generated by IEEEtran.bst, version: 1.14 (2015/08/26)
\begin{thebibliography}{10}
\providecommand{\url}[1]{#1}
\csname url@samestyle\endcsname
\providecommand{\newblock}{\relax}
\providecommand{\bibinfo}[2]{#2}
\providecommand{\BIBentrySTDinterwordspacing}{\spaceskip=0pt\relax}
\providecommand{\BIBentryALTinterwordstretchfactor}{4}
\providecommand{\BIBentryALTinterwordspacing}{\spaceskip=\fontdimen2\font plus
\BIBentryALTinterwordstretchfactor\fontdimen3\font minus
  \fontdimen4\font\relax}
\providecommand{\BIBforeignlanguage}[2]{{%
\expandafter\ifx\csname l@#1\endcsname\relax
\typeout{** WARNING: IEEEtran.bst: No hyphenation pattern has been}%
\typeout{** loaded for the language `#1'. Using the pattern for}%
\typeout{** the default language instead.}%
\else
\language=\csname l@#1\endcsname
\fi
#2}}
\providecommand{\BIBdecl}{\relax}
\BIBdecl

\bibitem{Richardson08Coding}
T.~Richardson and R.~Urbanke, \emph{{Modern Coding Theory}}.\hskip 1em plus
  0.5em minus 0.4em\relax Cambridge University Press, 2008.

\bibitem{xiaoqiang13ADMMLDPC}
G.~Zhang, R.~Heusdens, and W.~B. Kleijn, ``{Large Scale LP Decoding with Low
  Complexity},'' \emph{IEEE Communications Letters}, vol.~17, no.~11, pp.
  2152--2155, 2013.

\bibitem{Boyd06gossip}
S.~Boyd, A.~Ghosh, B.~Prabhakar, and D.~Shah, ``{Randomized Gossip
  Algorithms},'' \emph{IEEE Trans. Information Theory}, vol.~52, no.~6, pp.
  2508--2530, 2006.

\bibitem{Sontag11ML}
D.~Sontag, A.~Globerson, and T.~Jaakkola, ``{Introduction to Dual Decomposition
  for Inference},'' in \emph{Optimization for Machine Learning}.\hskip 1em plus
  0.5em minus 0.4em\relax MIT Press, 2011.

\bibitem{Joshi13ADMM}
S.~Joshi, M.~Codreanu, and M.~Latva-aho, ``{Distributed SINR balancing for MISO
  downlink systems via the alternating direction method of multipliers},'' in
  \emph{Proc. 11th Int. Symp. Modeling Optim. Mobile, Ad Hoc Wireless Networks
  (WiOpt)}, 2013, pp. 318--325.

\bibitem{Yang11ADMM}
J.~Yang and Y.~Zhang, ``{Alternating direction algorithms for l1-problems in
  compressive sensing},'' \emph{SIAM J. Sci. Comput.}, vol.~33, no.~1, pp.
  250--278, 2011.

\bibitem{Zhu17ADMM}
Y.~Zhu, ``{An Augmented ADMM Algorithm With Application to the Generalized
  Lasso Problem},'' \emph{Journal of Computational and Graphical Statistics},
  vol.~26, no.~1, pp. 195--204, 2017.

\bibitem{Figueiredo10ADMM}
M.~Figueiredo and J.~Bioucas-Dias, ``{Restoration of poissonian images using
  alternating direction optimization},'' \emph{IEEE Trans. Image Processing},
  vol.~19, no.~12, pp. 3133--3145, 2010.

\bibitem{Joshi13ADMMMISO}
S.~Joshi, M.~Codreanu, and M.~Latva-aho, ``{Distributed SINR balancing for MISO
  downlink systems via the alternating direction method of multipliers},'' in
  \emph{Proceedings of 11th International Symposium on Modeling \& Optimization
  in Mobile, Ad Hoc \& Wireless Networks (WiOpt)}, May 2013, pp. 318--325.

\bibitem{Teixeira16ADMM}
A.~Teixeira, E.~Ghadimi, I.~Shames, H.~Sandberg, and M.~Johansson, ``{The ADMM
  Algorithm for Distributed Quadratic Problems: Parameter Selection and
  Constraint Preconditioning},'' \emph{IEEE Trans. Signal Processing}, vol.~64,
  no.~2, pp. 290--305, 2016.

\bibitem{Ghadimi15ADMM}
E.~Ghadimi, A.~Teixeira, I.~Shames, and M.~Johansson, ``{Optimal Parameter
  Selection for the Alternating Direction Method of Multipliers (ADMM):
  Quadratic Problems},'' \emph{IEEE Trans. Automatic Control}, vol.~60, no.~3,
  pp. 644--658, 2015.

\bibitem{Giselsson17ADMM}
P.~Giselsson and S.~Boyd, ``{Linear convergence and metric selection in
  Douglas-Rachford splitting and ADMM},'' \emph{IEEE Trans. Automatic Control},
  vol.~62, no.~2, pp. 532--544, 2017.

\bibitem{Zhang16PDMM}
G.~Zhang and R.~Heusdens, ``{Distributed Optimization using the Primal-Dual
  Method of Multipliers},'' appearing in IEEE Trans. Signal and Information
  Processing over Networks, 2017.

\bibitem{Shi14ADMM}
W.~Shi, Q.~Ling, K.~Yuan, G.~Wu, and W.~Yin, ``{On the Linear Convergence of
  the ADMM in Decentralized Consensus Optimization},'' \emph{IEEE Trans. Signal
  Processing}, vol.~7, pp. 1750--1761, 2014.

\bibitem{Heming15Thesis}
H.~M. Zhang, ``{Distributed Convex Optimization: A Study on the Primal-Dual
  Method of Multipliers},'' Master's thesis, Delft University of Technology,
  2015.

\bibitem{xiaoqiang16BiADMM}
G.~Zhang and R.~Heusdens, ``{On Simplifying the Primal-Dual Method of
  Multipliers},'' in \emph{Proc. of IEEE International Conference on Acoustics,
  Speech, and Signal Processing (ICASSP)}, March 2016, pp. 4826--4830.

\bibitem{Tavakoli16PDMM}
V.~M. Tavakoli, J.~R. Jensen, R.~Heusdens, J.~Benesty, and M.~G. Christensen,
  ``{Ad hoc microphone array beamforming using the primal-dual method of
  multipliers},'' in \emph{Proc. Euro- pean Signal Processing Conf.}, 2016, pp.
  1088--1092.

\bibitem{Tavakoli17PDMM}
------, ``{Distributed max-SINR speech enhancement with ad hoc microphone
  arrays},'' in \emph{ICASSP}, 2017, pp. 151--155.

\bibitem{Connor16PDMM}
M.~O. Connor, W.~B. Kleijn, and T.~Abhayapala, ``{Distributed sparse MVDR
  beamforming using the bi-alternating direction method of multipliers},'' in
  \emph{Proc. of IEEE International Conference on Acoustics, Speech, and Signal
  Processing (ICASSP)}, 2016, pp. 106--110.

\bibitem{Sherson16LCMV_conf}
T.~Sherson, W.~B. Kleijn, and R.~Heusdens, ``{A Distributed Algorithm for
  Robust LCMV Beamforming},'' in \emph{Proc. of IEEE International Conference
  on Acoustics, Speech, and Signal Processing (ICASSP)}, 2016, pp. 101--105.

\bibitem{Sherson17PDMM}
{T. Sherson and R. Heusdens, W. B. Kleijn}, ``{Derivation and Analysis of the
  Primal-Dual Method of Multipliers Based on Monotone Operator Theory},''
  arXiv:1706.02654 [math.OC], 2017.

\bibitem{Ryu16Mono}
E.~K. Ryu and S.~Boyd, ``{Primer on monotone operator methods},'' \emph{Appl.
  Comput. Math}, vol.~15, no.~1, pp. 3--43, 2016.

\bibitem{Eckstein_phdThesis}
J.~Eckstein, ``{Splitting methods for monotone operators with applications to
  parallel optimization},'' Ph.D. dissertation, Massachusetts Institute of
  Technology, 1989.

\bibitem{Kalman60}
R.~E. Kalman, ``{A New Approach to Linear Filtering and Prediction Problems},''
  \emph{Journal of Basic Engineering}, vol.~82, pp. 35--45, 1960.

\bibitem{Kalman61newresults}
R.~E. Kalman and R.~S. Bucy, ``{New Results in Linear Filtering and Prediction
  Theory},'' \emph{Trans. ASME, Ser. D, J. Basic Eng}, p. 109, 1961.

\bibitem{KailathBook}
T.~Kailath, A.~Sayed, and B.~Hassibi, \emph{{Linear Estimation}}.\hskip 1em
  plus 0.5em minus 0.4em\relax New Jersey: Prentice Hall, 2000.

\bibitem{Auger13Kalman}
F.~Auger, M.~Hilairet, J.~M. Guerrero, E.~Monmasson, T.~Orlowska-Kowalska, and
  S.~Katsura, ``{Industrial applications of the Kalman filter: A review},''
  \emph{IEEE Trans. Ind. Electron.}, vol.~60, no.~12, pp. 5458--5471, 2013.

\bibitem{Olfati07DKF}
R.~Olfati-Saber, ``{Distributed Kalman Filtering for Sensor Networks},'' in
  \emph{Proc. IEEE Conf. on Decision and Control}, 2007, pp. 5492--5498.

\bibitem{xiaoqiang12LiCDQO}
Y.~Zeng and R.~Heusdens, ``{Linear Coordinate-Descent Message-Passing for
  Quadratic Optimization},'' \emph{Neural Computation}, vol.~24, no.~12, pp.
  3340--3370, 2012.

\bibitem{Moallemi09GaBP}
C.~C. Moallemi and B.~V. Roy, ``{Convergence of Min-Sum Message Passing for
  Quadratic Optimization},'' \emph{IEEE Trans. Inf. Theory}, vol.~55, no.~5,
  pp. 2413--2423, 2009.

\bibitem{Moallemi10BPConvex}
------, ``{Convergence of Min-Sum Message Passing for Convex Optimization},''
  \emph{IEEE Trans. Inf. Theory}, vol.~56, no.~4, pp. 2041--2050, 2010.

\end{thebibliography}

\end{document}